\documentclass[11pt,oneside,onecolumn,reqno]{amsart}

\usepackage{amssymb,latexsym,
amsxtra,amscd,amsfonts,amsthm,amsmath,verbatim,epsfig,xypic}

\usepackage{color}

\usepackage[colorlinks,pagebackref=true]{hyperref}

\newcommand{\arxiv}[1]{\href{http://arxiv.org/abs/#1}{{\tt arXiv:#1}}}

\makeatletter

\def\@citecolor{blue}
\def\@linkcolor{blue}
\def\@urlcolor{blue}

\def\im{\operatorname{im}}
\def\image{\operatorname{image}}
\def\ker{\operatorname{ker}}
\def\dim{\operatorname{dim}}
\def\depth{\operatorname{depth}}

\def\ass{\operatorname{Ass}}
\def\supp{\operatorname{Supp}}

\def\degree{\operatorname{degree}}

\def\ZZ{\mathbb Z}
\newcommand{\m}{\mathfrak m}
\newcommand{\p}{\mathfrak p}
\newcommand{\ov}{\overline}
\newcommand{\lm}{{\lambda}}
\newcommand{\R}{{\mathcal R}}
\newcommand{\F}{{\mathcal F}}
\newcommand{\NN}{{\mathbb N}}
\theoremstyle{plain}

\newtheorem{theorem}{Theorem}[section]
\newtheorem{lemma}[theorem]{Lemma}
\newtheorem{proposition}[theorem]{Proposition}
\newtheorem{corollary}[theorem]{Corollary}

\theoremstyle{definition}

\newtheorem{remark}[theorem]{Remark}

\newtheorem{example}[theorem]{Example}

\begin{document}

 \title[ ] {Local Cohomology of Multi-Rees algebras with \\ [2mm] applications to joint reductions and complete ideals}
\author{Shreedevi K. Masuti}
 \address{Institute of Mathematical Sciences, 
IV Cross Road, CIT Campus, 
Taramani, 
Chennai 600113, 
Tamil Nadu, India.}
\email{shreedevikm@imsc.res.in}
 \thanks{Key words : strict complete reductions, good complete reductions, good joint reductions, 
 multi-Rees algebra, local cohomology, complete ideal, normal Hilbert polynomial, 
 normal joint reduction number zero. 
}
\author{Tony J. Puthenpurakal}
\address{Department of Mathematics, Indian Institute of Technology
Bombay, Powai, Mumbai, India - 400076}
\email{tputhen@math.iitb.ac.in}
\author{J. K. Verma}
 \address{Department of Mathematics, Indian Institute of Technology
Bombay, Powai, Mumbai, India - 400076}
\email{jkv@math.iitb.ac.in}
\maketitle
\begin{center}
 {\it Dedicated to Professor Ngo Viet Trung on his sixtieth birthday}
\end{center}

\begin{abstract}
Let $(R,\m)$ be a Cohen-Macaulay local ring of dimension $d$ and ${\bf I}=(I_1,\ldots,I_d)$ be 
$\m-$primary ideals in $R$. We prove that $\lm_R([H^d_{(x_{ii}t_i:1\leq i \leq d)}(\R(\F)]_{\bf n})< \infty$, 
for all ${\bf n} \in \mathbb N^d$, where $\F=\{\F({\bf n}):{\bf n}\in \mathbb Z^d\}$ is an
${\bf I}-$admissible filtration and $(x_{ij})$ is a strict complete reduction of $\F$ and $\R(\F)$ is the multi-Rees algebra of $\F.$
As a consequence we prove that the normal joint reduction 
number of $I,J,K$ is zero in an analytically unramified Cohen-Macaulay local ring of 
dimension $3$ if and only if 
$\ov{e}_3(IJK)-[\ov{e}_3(IJ)+\ov{e}_3(IK)+\ov{e}_3(JK)]+\ov{e}_3(I)+\ov{e}_3(J)+\ov{e}_3(K)=0$. 
This generalizes a theorem of Rees on joint reduction number zero in dimension $2$. 
We apply this theorem to generalize a theorem of M. A. Vitulli in dimension $3.$
\end{abstract}

\section{Introduction}
The objective of this paper is to explore  finite generation of multigraded components of 
local cohomology modules of extended multi-Rees algebras of $\ZZ^d$-graded filtrations of ideals in a 
$d$-dimensional Cohen-Macaulay local ring. As a consequence we obtain a satisfactory 
generalisation in dimension $3$ of an important theorem of D. Rees about $\NN^2$-graded filtrations of 
complete ideals of joint reduction number zero 
in two-dimensional analytically unramified Cohen-Macaulay local rings \cite{rees}.
As a consequence of his theorem, Rees reproved Zariski's theorem about products
of complete ideals in two-dimensional regular and psuedo-rational local rings and he also obtained
a formula for the normal Hilbert polynomial of a complete ideal in a pseudo-rational
local ring. An analogue of Rees' theorem in any dimension is not yet known. This paper
presents an analogue in dimension $3.$

We shall apply our finiteness theorem to provide necessary and sufficient conditions for 
the integral closure filtration $\{\overline{I^rJ^tK^t}\}$ of three $\m$-primary ideals 
in a $3$-dimensional analytically unramified Cohen-Macaulay local ring to have joint 
reduction number zero. This, in turn, gives us  sufficient conditions for power 
products of complete monomial ideals to be compete in $3$-dimensional polynomial rings. 
This yields a generalization of a theorem of M. A. Vitulli \cite{v} in dimension $3.$ 

We now describe our main results in detail for which we need to recall several definitions and 
introduce necessary notation. Let $(R,\m)$ be a $d$-dimensional Noetherian local ring and let 
${\bf I}=(I_1,I_2, \ldots,I_s)$ for  ideals 
$I_1, I_2, \ldots,I_s$ of $R$. For ${\bf n}=(n_1,\ldots,n_s) \in \mathbb Z^s$, let ${\bf I}^{\bf n}=
I_1^{n_1}\ldots I_s^{n_s}$. An {\it ${\bf I}-$filtration} of ideals is a set of ideals 
$\mathcal F=\{\mathcal F({\bf n}): {\bf n} \in \mathbb Z^s\}$ such that
\begin{enumerate}
 \item $\mathcal F({\bf n}) \subseteq \mathcal F({\bf m})$ for all ${\bf n}, {\bf m}
 \in \mathbb Z^s$ and ${\bf n}\geq {\bf m},$
 \item ${\bf I}^{\bf n} \subseteq \mathcal F({\bf n})$ for all $\bf{n} \in \ZZ^s,$
\item $\mathcal F({\bf n})\mathcal F({\bf m}) \subseteq \mathcal F({\bf n}+{\bf m})$ 
for all ${\bf m},{\bf n} \in \mathbb Z^s.$
 \end{enumerate}
 
 The principal examples of $\ZZ^s$-graded filtrations are the {\em adic filtration}
 $\{{\bf{I}}^{\bf n}= I_1^{n_1}I_2^{n_2}\ldots I_s^{n_s}\}$ 
 and
 the {\em integral closure filtration} $\{\ov{\bf{I}^{\bf n}}= \ov{I_1^{n_1}I_2^{n_2} \ldots I_s^{n_s}}\}.$  
 Recall that the integral closure of an ideal $J$ in a ring $R$ denoted by $\ov{J}$ is the ideal
 $$\{r \in R\mid r^n+a_1r^{n-1}+a_2r^{n-2}+\cdots+a_n=0 \;\mbox{ for some } a_i \in 
 J^i \;\; \mbox{ for } i=1,2,\ldots,n\}.$$
 The ideal  $J$ is called {\em complete} or {\em integrally closed}  if $J=\ov{J}.$ 
 For indeterminates $t_1,\ldots,t_s$ over $R$, put
 ${\bf t}^{\bf n}=t_1^{n_1}\ldots t_s^{n_s}$. 
 We say $\mathcal F$ is an ${\bf I}-${\it admissible filtration} if the extended Rees 
 algebra $\mathcal R^\prime(\mathcal F)=\bigoplus_{{\bf n}\in \mathbb Z^s}\mathcal F({\bf n})
 {\bf t}^{\bf n}$ is a finite module over the extended Rees algebra 
 $\mathcal R^\prime({\bf I})=
 \bigoplus_{{\bf n} \in \mathbb Z^s}{\bf I}^{\bf n}{\bf t}^{\bf n}$. The integral closure
 of  ${\mathcal R}^\prime({\bf I})$ in the ring $R[t_1^{\pm 1},t_2^{\pm 2},\ldots,t_s^{\pm 1}] $ is 
 the extended Rees algebra $\ov{\mathcal R^\prime}({\bf I})=\mathcal R^\prime ({\mathcal F})$ where 
 $\mathcal F({\bf n})=\ov{{\bf I}^{\bf n}}$ for all ${\bf n} \in \mathbb Z^s$. 
 
 We write ${\bf e}_i$ for the vector in $\ZZ^s$ which has $1$ in the $i^{th}$ position and $0$ elsewhere for $i=1,2,\ldots,s.$ The zero vector in $\ZZ^s$ will be denoted by $\bf 0.$
 Set ${\bf e}={\bf e}_1+{\bf e}_2+\cdots+{\bf e}_s.$
Now assume that $I_1,I_2, \ldots, I_s$ are $\m$-primary ideals. An $s\times d$ matrix $(x_{ij})$ where $x_{ij} \in I_i$ for all $j=1,\ldots,d$ and 
$i=1,\ldots,s$ is called a {\it complete reduction} of the ${\bf I}-$filtration $\mathcal F$ if $$\mathcal F({{\bf n}+{\bf e}})=(y_1,\ldots,y_d)\mathcal F({\bf n}) $$
for all large 
${\bf n} \in \mathbb N^s.$
Here $y_j=\prod_{i=1}^sx_{ij}$ for $j=1,\ldots,d.$ 
We say that $(x_{ij})$ is a {\it good complete reduction} of $\F$ if in addition $(x_{ij})$ satisfies 
\begin{eqnarray*}
 (y_j:j \in A) \cap \F({\bf n})=(y_j:\in A)\F({\bf n}-{\bf e}) \mbox{ for }{\bf n} \geq |A|{\bf e} 
\end{eqnarray*}
for every proper subset $A \subsetneq \{1,\ldots,d\}$. Here $|A|$ denotes the cardinality 
of the set $A$. If $s=d$ then the elements $x_{ii}=a_i$ for $i=1,2,\ldots,d$ satisfy the equation
$$\mathcal F({\bf n})=a_1{\mathcal F}({\bf n}-{\bf e}_1)+a_2 {\mathcal F}({\bf n}-{\bf e}_2)+\cdots+
a_d {\mathcal F}({\bf n}-{\bf e}_d)\mbox{ for }{\bf n} \gg {\bf 0}.$$
If the above equation holds for $a_i \in I_i$ for $i=1,2,\ldots,d$ then we say that $(a_1,a_2,\ldots, a_d)$ constitutes a {\em joint reduction } of $\mathcal F.$  If the above equation holds for all $\bf n \geq e$ then we say that the filtration $\mathcal F$ has {\em joint reduction number zero } with respect to the joint reduction $(a_1,a_2,\ldots,a_d).$
If in addition, for each proper subset $A $ of $\{1,2,\ldots,d\}$ and  ${\bf n}\geq 
\sum_{i \in A}{\bf e}_i$
$$(a_i : i \in A)\cap \mathcal F({\bf n})=\sum_{ i \in A}  a_i\mathcal F({\bf n}-{\bf e}_i)$$
then we say that $(a_1,a_2,\ldots, a_d)$  is a {\em good joint reduction of $\mathcal F.$}
We say that $(x_{ij})$ is a {\em strict complete reduction} of 
$\mathcal F$ if for all ${\bf n }\geq j{\bf e}$ and for $j=1,2,\ldots,d-1$ we have
$$(y_1,y_2,\ldots,y_j)\cap \mathcal F({\bf n})=(y_1,y_2,\ldots,y_j) \mathcal F({\bf n-e}).$$
 One of the main observations in this paper is that joint  reductions can be studied using local cohomology  with support in the 
 ideal  $\mathcal I=(a_1t_1,a_2t_2,\ldots,a_dt_d)$ of the extended Rees algebra $\mathcal R'(\mathcal F).$ Our main result in section $2$ is the following finiteness theorem:

\begin{theorem} [\bf Finiteness Theorem] Let $(R,\m)$ be a Cohen-Macaulay local ring of dimension $d.$
Let $\mathcal F$ be an ${\bf I}-$admissible filtration of $\m$-primary ideals and $M=\mathcal R^\prime(\mathcal F)$. Then 
\begin{enumerate}
\item  $[H^i_{\mathcal I}(M)]_{\bf n}=0 \mbox{ for all }{\bf n} \gg {\bf 0}$ and for all $i \geq 0.$ 

\item If $(y_1,y_2,\ldots,y_d)$ is a strict complete reduction of $\mathcal F$ then
$\lambda_R\left([H^d_{\mathcal I}(M)]_{\bf n}\right) < \infty$ for all $\bf n \geq 0.$
\end{enumerate}
\end{theorem}

In view of the finiteness theorem, it is natural to ask under what conditions, strict complete reductions and good joint reductions exist. Rees proved their existence for the 
integral closure filtration $\{\overline{I^rJ^s}\}$ of $\m$-primary ideals $I,J$
in any $2$-dimensional analytically unramified Cohen-Macaulay local ring with infinite residue field. We show their existence in dimension $3$ under additional hypothesis.
For an $\bf I$-admissible filtration $\mathcal F$ we set 
$$G_i(\mathcal F)=\bigoplus_{\bf n \in \mathbb N^s} \frac{\mathcal F(\bf n)}{\mathcal F({\bf n}+{\bf e}_i)}\;\;\mbox{ for all } i=1,2,\ldots,s.$$

For any $\ZZ^s$ or $\NN^s$-graded algebra $R$ we write $R_{++}$ for the ideal
generated by all elements of degree at least ${\bf e}.$
\begin{theorem} Let $(R,\m)$ be a three dimensional Cohen-Macaulay local ring
with infinite residue field and ${\bf I}=(I_1,I_2,I_3)$ for $\m$-primary ideals 
$I_1,I_2,I_3$ in $R$. Suppose
$\mathcal F$ is an $\bf I$-admissible filtration. Suppose that $G_i(\mathcal F)_{++}$ 
has positive depth and $[H^1_{G_i(\mathcal F)_{++}}(G_i(\mathcal F)]_{\bf n}=0$
for all $i=1,2,3$ and all ${\bf n}+{\bf e}_i \geq 0.$ Then there exists a good complete reduction
$(x_{ij})$ of $\mathcal F$ such that $(x_{i,\sigma(i)} : i=1,2,3)$ is a good joint reduction of $\mathcal F$ and for any proper subset $A$ of $\{1,2,3\}$ and $i=1,2,3$
$$ (x_{ij}: j \in A)\cap \mathcal F({\bf n})= (x_{ij}: j \in A) \mathcal F({\bf n}-{\bf e}_i)\;\; 
\mbox{ for all } {\bf n}-{\bf e}_i \geq 0.$$

\end{theorem}

\noindent
The principal application of  good joint reductions is a formula for the length of 
$[H^3_{\mathcal I}(\overline{\mathcal R'}(\mathcal F))]_{(0,0,0)} .$ In order to state the next theorem proved in this paper, we recall the notion of the normal Hilbert function and the normal Hilbert polynomial.
For an $\m$-primary ideal $I$ in an analytically unramified local ring $(R,\m)$ of dimension $d$, 
the {\em normal Hilbert function of $I$} is the function $\overline{H}_I(n)=
\lambda(R/\overline{I^n}).$  There exists a polynomial
$\overline{P}_I(x)$ of degree $d$ such that  $\overline{P}_I(n)=\overline{H}_I(n)$ for all large $n.$ We write this polynomial in the form
$$\overline{P}_I(x)=\ov{e}_0(I)\binom{x+d-1}{d}-\ov{e}_1(I)\binom{x+d-2}{d-1}+\cdots+
(-1)^d\ov{e}_d(I).$$
The polynomial $\ov{P}_I(x)$ is called the {\em normal Hilbert polynomial of $I.$} 
Let $\ov{\mathcal R^\prime}:=\ov{\R^\prime}(I,J,K)$.

\begin{theorem} \label{main} Let $(R,\m)$ be a $3$-dimensional Cohen-Macaulay analytically unramified local ring. 
Assume that a good complete reduction of the filtration $\{\ov{I^rJ^sK^t}\}$ exists. 
Let $(a,b,c)$ be a good joint reduction of the filtration $\{\ov{I^rJ^sK^t}\}.$ Then
$$\lm_R([H^3_{(at_1,bt_2,ct_3)}(\ov{\mathcal R^\prime})]_{(0,0,0)})=
\ov{e}_3(IJK)-[\ov{e}_3(IJ)+\ov{e}_3(IK)+\ov{e}_3(JK)]+\ov{e}_3(I)+\ov{e}_3(J)+\ov{e}_3(K).$$
\end{theorem}

\noindent
As a consequence of this theorem, we are able to generalise Rees' Theorem \cite{rees} in dimension three. 
In dimension three we say that the {\it normal joint reduction number of $I,J,K$ is zero} with respect to $(a,b,c)$ if 
for all $r,s,t>0$
$$\ov{I^rJ^sK^t}=a\ov{I^{r-1}J^sK^t}+b\ov{I^rJ^{s-1}K^t}+c\ov{I^rJ^sK^{t-1}}.$$

\begin{theorem} Let the assumptions be as in Theorem \ref{main}.
Then following statements are equivalent:
\begin{enumerate}
\item $[H^3_{(at_1,bt_2,ct_3)}(\ov{\mathcal R^\prime})]_{(0,0,0)}=0,$
 \item The normal joint reduction number of $I,J,K$ is zero with respect to $(a,b,c),$
\item $\ov{e}_3(IJK)-[\ov{e}_3(IJ)+\ov{e}_3(IK)+\ov{e}_3(JK)]+\ov{e}_3(I)+\ov{e}_3(J)+\ov{e}_3(K)=0.$
\end{enumerate}
\end{theorem}
Finally we apply our criterion for joint reduction number zero to obtain a generalisation
of a theorem of M. A. Vitulli \cite{v} about normal monomial ideals.

\begin{theorem} Let $k$ be a field and $R=k[x,y,z],\m=(x,y,z).$ Suppose that $I,J,K$ are $\m$-primary monomial ideals of $R$ such that $I^rJ^sK^t$ is complete for all
$r+s+t \leq 2.$ Then $I^rJ^sK^t$ is complete for all $r,s,t \geq 0.$

\end{theorem}

\noindent {\bf Acknowledgement:} We thank Parangama Sarkar for a careful reading of the 
manuscript.

\section{Finiteness of multigraded components of local cohomology modules of Rees algebras}
In \cite{masuti-verma}, authors have derived a formula for 
$\lm_R([H^2_{(at_1,bt_2)}(\ov{\mathcal R^\prime}(I,J))]_{(r,s)})$ for all 
integers $r,s \geq 0$, which shows that $\lm_R([H^2_{(at_1,bt_2)}(\ov{\mathcal R^\prime}(I,J))]_
{(r,s)}) <\infty $ in an analytically unramified Cohen-Macaulay local ring of dimension $2$ 
for a good joint reduction $(a,b)$ of $\{\ov{I^rJ^s}\}$. Here 
$\ov{\mathcal R^\prime}(I,J)$ is the bigraded Rees algebra of the filtration $\{\ov{I^rJ^s}\}$. 
In general 
$\lm_R([H^2_{(at_1,bt_2)}(\mathcal R^\prime(\mathcal F))]_{(r,s)})$ need 
not be finite for an $(I,J)-$admissible filtration $\mathcal{F}$ and 
for a joint reduction $(a,b)$ of $\mathcal F$, even if $R$ is regular. In this section 
we give an example to illustrate this. However, we show that 
$\lm_R([H^d_{(a_1t_1,\ldots,a_dt_d)}(\mathcal R^\prime(\mathcal F))]_{\bf n}) <\infty$ for 
all ${\bf n} \in \mathbb N^d$ in a Cohen-Macaulay local ring of dimension $d$, where $a_i=x_{ii}$ for a good complete reduction $(x_{ij})$ 
of an ${\bf I}-$admissible filtration $\mathcal F$. 

\subsection{Preliminaries}
First we need some lemmas.
\begin{lemma} \label{preliminary lemma}
 Let $R=\bigoplus_{{\bf n} \in \mathbb Z^s} R_{\bf n}$ be a Noetherian 
 $\mathbb Z^s-$graded ring. Suppose $E=\bigoplus_{{\bf n}\in \mathbb Z^s} E_{\bf n}$ 
 is a $\mathbb Z^s-$graded $R-$module. Suppose $E_{\bf n}=0$ for all ${\bf n} \geq {\bf m} $.
 \begin{enumerate}
 \item \label{lemma part(a)} Let $f \in R_{\bf d}$, where ${\bf d} \geq {\bf 0}$. Then for all 
  $i \geq 0$ and all ${\bf n} \geq {\bf m}$,
  \begin{eqnarray}\label{part(a)}
   [H^i_{(f)}(E)]_{\bf n}=0 .
  \end{eqnarray}
\item Let $f_1,\ldots,f_r$ be homogeneous elements of $R$ of degrees ${\bf d_1},\ldots,{\bf d_r} 
\in \mathbb N^s$. 
Then for all $i \geq 0$ and all ${\bf n} \geq {\bf m},$
\begin{eqnarray*}
 [H^i_{(f_1,\ldots,f_r)}(E)]_{\bf n}=0 .
\end{eqnarray*}
\item Assume $R_{\bf 0}=A$ is a local Noetherian ring with maximal ideal $\m.$ If 
$\supp E_{\bf n} \subseteq \{\m\}$ for all ${\bf n} \in \mathbb Z^s$ then 
$\supp~ [H^i_{(f_1,\ldots,f_r)}(E)]_{\bf n} \subseteq \{\m\}$ for all ${\bf n} \in 
\mathbb Z^s$.
 \end{enumerate}
\end{lemma}
\begin{proof}
 \begin{enumerate}
  \item By \cite[Exercise 5.1.14]{bs}, we have an exact sequence
  $$0 \longrightarrow H^0_{(f)}(E) \longrightarrow E \longrightarrow E_f \longrightarrow 
  H^1_{(f)}(E) \longrightarrow 0.$$
  Since $E_{\bf n}=0$ for all ${\bf n} \geq {\bf m}$, $[H^0_{(f)}(E)]_{\bf n}=0$ 
  for all ${\bf n} \geq {\bf m}$. To show $[H^1_{(f)}(E)]_{\bf n}=0$ for all ${\bf n} 
  \geq {\bf m}$, consider any $z \in [E_f]_{\bf n}$. Then $z=\frac{b}{f^r}$ for some 
  $r \in \mathbb N$ and $b \in E$. Therefore $\degree b={\bf n}+r{\bf d} \geq {\bf n}$. 
  Hence $[E_f]_{\bf n}=0$ for all ${\bf n}\geq {\bf m}$. This gives $[H^1_{(f)}(E)]_{\bf n}=0$ 
  for all ${\bf n}\geq {\bf m}$.
  \item We use induction on $r.$ For $r=1$ the assertion follows from the equation (\ref{part(a)}). 
  Assume $r>1$ and the result is true for $r-1$. Let $I=(f_1,\ldots,f_{r-1})$. By 
  \cite[Exercise 5.1.14]{bs}, we have the exact sequence
  $$\cdots\longrightarrow H^{i-1}_{I}(E)\longrightarrow H^{i-1}_{I}(E)_{f_r} 
\longrightarrow   H^i_{(I,f_r)}(E) \longrightarrow H^i_I(E) \longrightarrow 
  H^i_{I}(E)_{f_r} \longrightarrow \cdots.$$
 For all $i \geq 0,$ we have an exact sequence
\begin{eqnarray} \label{relating i-1 and i local cohomology}
0\longrightarrow H^1_{(f_r)}(H^{i-1}_I(E)) \longrightarrow H^i_{(I,f_r)}(E) 
\longrightarrow H^0_{(f_r)}(H^{i}_I(E)) \longrightarrow 0.
\end{eqnarray}
By induction hypothesis 
$$[H^j_I(E)]_{\bf n}=0\mbox{ for all }{\bf n}\geq {\bf m} \mbox{ and all }j \geq 0.$$
Therefore from the equation (\ref{part(a)}), we obtain
$$[H^k_{(f_r)}(H^j_I(E))]_{\bf n}=0 \mbox{ for all }{\bf n}\geq {\bf m} \mbox{ and all } k,j\geq 0 
.$$
Hence $[H^i_{(I,f_r)}(E)]_{\bf n}=0$ for all ${\bf n} \geq {\bf m}$ and all $i \geq 0$.
\item Let $\mathfrak p \neq \m$ be a prime ideal of $A$. Note that $R_{\bf 0}\setminus \mathfrak{p}$ 
has elements of degree ${\bf 0}$. Since $(E_{\bf n})_{\mathfrak{p}}=0$ 
for all ${\bf n} \in \mathbb Z^s$, $E_\mathfrak p=0$. Hence 
$$(H^i_{(f_1,\ldots,f_r)}(E))_{\mathfrak{p}}=H^i_{(f_1,\ldots,f_r)A_{\mathfrak{p}}}
(E_{\mathfrak{p}})=0.$$
Therefore $([H^i_{(f_1,\ldots,f_r)}(E)]_{\bf n})_{\mathfrak{p}}=0$ for all 
${\bf n} \in \mathbb Z^s$ and hence 
$\supp ~[H^i_{(f_1,\ldots,f_r)}(E)]_{\bf n} \subseteq\{\m\}$.
\end{enumerate}
\end{proof}

\begin{corollary} \label{corollary of preliminary lemma}
 Let $R=\bigoplus_{{\bf n} \in \mathbb Z^s}R_{\bf n}$ be a Noetherian $\mathbb Z^s-$
 graded ring and let $M=\bigoplus_{{\bf n}\in \mathbb Z^s}$ be a finitely generated 
 $\mathbb Z^s-$graded $R-$module. Let $q$ be a graded ideal of 
 $R$. Suppose there exists ${\bf m} \in \mathbb Z^s$ such that for all $i \geq 0$ and all 
 ${\bf n} \geq {\bf m}$, 
 $$[H^i_q(M)]_{\bf n}=0 .$$
 Let $f_1,\ldots,f_r$ be homogeneous elements with $\degree f_i \in \mathbb N^s$ 
 for all $i$. Then for $q^\prime=(q,f_1,\ldots,f_r)$,
 $$[H^i_{q^\prime}(M)]_{\bf n}=0\mbox{ for all }{\bf n}\geq {\bf m} \mbox{ and all }i \geq  0.$$
Furthermore, if $R_0=A$ is a local Noetherian ring and $\lm_A([H^i_{q}(M)]_{\bf n})< \infty $
for all ${\bf n} \in \mathbb Z^s$ and all $i \geq 0$ then 
\begin{eqnarray}\label{support in q'}
 \supp [H^i_{q^\prime}(M)]_{\bf n} \subseteq \{\m\} \mbox{ for all }{\bf n} \in \mathbb Z^s 
 \mbox{ and all }i \geq 0.
\end{eqnarray}
 \end{corollary}
\begin{proof}
 It suffices to show the $r=1$ case. Let $q^\prime=(q,f)$ where $f=f_1$. 
 From (\ref{relating i-1 and i local cohomology}), we have the exact sequence
 \begin{eqnarray}\label{relating i-1 and i local cohomology in corollary}
 0\longrightarrow H^1_{(f)}(H^{i-1}_q(M)) \longrightarrow H^i_{q^\prime}(M) 
\longrightarrow H^0_{(f)}(H^{i}_q(M)) \longrightarrow 0.
\end{eqnarray}
By Lemma \ref{preliminary lemma}(1), $[H^i_{q^\prime}(M)]_{\bf n}=0$ 
for all ${\bf n}\geq {\bf m}$ and all $i \geq 0$. To prove (\ref{support in q'}), 
localize the exact sequence (\ref{relating i-1 and i local cohomology in corollary}) 
at a prime ideal $\mathfrak{p}$ of $A$, $\mathfrak p \neq \m$, and 
use Lemma \ref{preliminary lemma}(3) to obtain 
$H^i_{q^\prime}(M)_\mathfrak p=0$ which gives $\supp ~ [H^i_{q^\prime}(M)]_{\bf n} \subseteq \{\m\} 
$ for all ${\bf n} \in \mathbb Z^s$.
\end{proof}

\subsection{Local Cohomology modules of extended Rees algebras of a $\mathbb Z^d-$filtration 
of ideals}
Let $(R,\m)$ be a $d-$dimensional Cohen-Macaulay local ring with infinite residue 
field in this subsection. Let ${\bf I}=(I_1,\ldots,I_d)$ for $\m-$primary ideals 
$I_1,\ldots,I_d$ of $R$. Let the $d \times d$ matrix $(x_{ij})$ where $x_{ij} \in I_i$ for all $j=1,\ldots,d$ and 
$i=1,\ldots,d$ be  a complete reduction of the ${\bf I}-$filtration $\mathcal F.$
Let  $y_j=\prod_{i=1}^dx_{ij}$ for $j=1,\ldots,d$. 
Let $\mathcal R({\bf I})=\bigoplus_{{\bf n} \in \mathbb N^d}{\bf I}^{\bf n}{\bf t}^{\bf n}$ 
be the Rees algebra of the filtration $\{{\bf I}^{\bf n}\}$ and 
$\mathcal R({\bf I})_{++}=\bigoplus_{{\bf n} \geq {\bf e}} {\bf I}^{\bf n}{\bf t}^{\bf n}$ 
be an ideal in $\R({\bf I})$. 
Then $q=(y_1{\bf t}^{\bf e},\ldots,y_d{\bf t}^{\bf e}) \subseteq \mathcal R({\bf I})_{++}$. 
It is clear that $\sqrt q =\sqrt {\mathcal R({\bf I})_{++}}$. Hence 
$$H^i_{\mathcal R({\bf I})_{++}}(\mathcal R^\prime(\mathcal F))=H^i_q(\mathcal R^\prime(\mathcal F))$$
for all $i \geq 0$. By an easy argument the following result can be derived from 
\cite[Theorem 5.1]{jayanthan-verma}. We set $\mathcal R_{++}=\mathcal R({\bf I})_{++}$.

\begin{proposition} \label{difference formula for filtration}
Let $\mathcal F$ be an admissible ${\bf I}-$filtration. Let $M=\mathcal R^\prime(\mathcal F)$ 
and $\mathcal R^\prime=\mathcal R^\prime({\bf I})$. Then 
\begin{enumerate}
 \item $[H^i_{\mathcal R_{++}}(M)]_{\bf n}=0$ for all ${\bf n} \gg {\bf 0}$.
 \item $\lm_R([H^i_{\mathcal R_{++}}(M)]_{\bf n}) < \infty$ for all ${\bf n} \in \mathbb Z^d$.
\end{enumerate}
\end{proposition}

Put $a_i=x_{ii}$ for $i=1,\ldots,d$ and $\mathcal I=(a_1t_1,\ldots,a_dt_d)$. Then 
$\mathcal R_{++} \subseteq \sqrt \mathcal I$ and for all $i \geq 0$
$$H^i_{(\mathcal I,\mathcal R_{++})}(M)=H^i_{(\mathcal I,q)}(M)=H^i_{\mathcal I}(M).$$

\begin{corollary} \label{vanishing of lc for large n} 
Let $\mathcal F$ be an ${\bf I}-$admissible filtration and $M=\mathcal R^\prime(\mathcal F)$. Then
$$[H^i_{\mathcal I}(M)]_{\bf n}=0 \mbox{ for all }{\bf n} \gg {\bf 0}.$$
\end{corollary}
\begin{proof}
 Follows from Corollary \ref{corollary of preliminary lemma} and Proposition 
 \ref{difference formula for filtration}.
\end{proof}

Let $\mathcal R(\mathcal F)=\bigoplus_{{\bf n} \in \mathbb N^g}\mathcal F({\bf n}){\bf t}^{\bf n}$ 
be the Rees algebra of $\mathbb N^g-$graded filtration $\mathcal F$.  

\begin{corollary} \label{vanishing of lc of Nj for large n}
Let $\mathcal F$ be an ${\bf I}-$admissible filtration and $M=\mathcal R^\prime(\mathcal F)$. 
Put $\theta_j=y_jt_1\ldots t_d$ for 
$j=1,\ldots,d$, $N_j=M/(\theta_1,\ldots,\theta_j)M$. 
Then
\begin{enumerate}
 \item $[H^i_{\mathcal R_{++}}(N_j)]_{\bf n}=0$ for all ${\bf n} \gg {\bf 0} $, 
 $j=1,\ldots,d$ and all $i \geq 0.$
 \item $[H^i_{\mathcal I}(N_j)]_{\bf n}=0$ for all ${\bf n} \gg {\bf 0}$, $j=1,\ldots, d$ and all 
 $i \geq 0$.
\end{enumerate}
\end{corollary}
\begin{proof}
 \begin{enumerate}
  \item Let $(N_j)_{\geq {\bf 0}}=  \bigoplus_{{\bf n} \geq {\bf 0}}[N_j]_{\bf n}$. 
  Since $\mathcal R(\mathcal F)$ is a finite $\mathcal R({\bf I})-$module, $(N_j)_{\geq {\bf 0}} $ 
  is a finite $\mathcal R({\bf I})-$module for all $j=1,\ldots,d$. Hence, 
  from \cite[Theorem 2.3]{jayanthan-verma}, 
  $[H^i_{\mathcal R_{++}}((N_j)_{\geq {\bf 0}})]_{\bf n}=0$ for ${\bf n}\gg {\bf 0}$. 
  Therefore, by \cite[Proposition 4.5]{jayanthan-verma}, 
  $[H^i_{\mathcal R_{++}}(N_j)]_{\bf n}=0$ for ${\bf n}\gg {\bf 0}$.
  
  \item Follows from part(1) and Corollary \ref{corollary of preliminary lemma}.
 \end{enumerate}
\end{proof}

\begin{remark} \label{remark on gcr}
\begin{enumerate}
 \item Every good complete reduction of $\F$ is strict good complete reduction of $\F$. 

\item \label{jr coming from gcr are good} If $(x_{ij})$ is a good complete reduction of the filtration $\mathcal F$ then $(x_{i,\sigma(i)})_{i \in S_d}$ is a 
 good joint reduction for every permutation $\sigma \in S_d$. Here $S_d$ is the symmetric group on $d$ letters.
\end{enumerate}
\end{remark}

\begin{lemma} \label{vanishing of lc of Nj}
 Let $(y_1,\ldots,y_d)$ be a strict complete reduction of $\mathcal F$. Then 
 for $j=1,\ldots,d-1$,
 \begin{eqnarray*}
 \lm_R[H^{d-j}_{\mathcal I}(N_j)]_{\bf n}<\infty \mbox{ for all }{\bf n}\geq j{\bf e}.
 \end{eqnarray*}
\end{lemma}
\begin{proof}
 Apply descending induction on $j$. Suppose $j=d-1$. Consider the exact sequence 
 \begin{equation}\label{exact sequence relating N(d-1) and N(d)}
\diagram
 0\rto& B^{(d-1)}\rto&N_{d-1}(-{\bf e})\drto \rrto^{\theta_{d}}&& N_{d-1}\rto& N_d \rto &0\\
  &&&K^{(d-1)}\drto\urto&&&\\
 &&0\urto&&0&&
\enddiagram
\end{equation}
Note that 
\begin{eqnarray*}
 [N_{d-1}]_{\bf n}=\frac{\mathcal F({\bf n})}{(y_1,\ldots,y_{d-1})\mathcal F({\bf n}-{\bf e})}.
\end{eqnarray*}
We show that $[H^i_{\mathcal I}(B^{(d-1)})]_{\bf n}=0$ for all ${\bf n}\geq d{\bf e}$ and 
all $i \geq 0.$ In view of Lemma \ref{preliminary lemma}, it is enough to show that 
$[B^{(d-1)}]_{\bf n}=0$ for all ${\bf n}\geq d {\bf e}$. Let $\ov{p} \in [B^{(d-1)}]_{\bf n}$ 
for some $p \in \mathcal F({\bf n}-{\bf e})$, where $\ov{p}$ denotes the image of $p$ in 
$N_{d-1}(-{\bf e})$. 
Then $y_d p \in (y_1,\ldots,y_{d-1})\mathcal F({\bf n}-{\bf e})$. Since $y_1,\ldots,y_d$ is a 
regular sequence 
$$p \in (y_1,\ldots,y_{d-1})\cap \mathcal F({\bf n}-{\bf e})=
(y_1,\ldots,y_{d-1})\mathcal F({\bf n}-2{\bf e}) \mbox{ for all }{\bf n} \geq d{\bf e}.$$
Hence $\ov p=0$. Thus $[B^{(d-1)}]_{\bf n}=0$ for all ${\bf n} \geq d{\bf e}$. 
By the exact sequence (\ref{exact sequence relating N(d-1) and N(d)}), we get that the sequence 
\begin{eqnarray*}
 \cdots\longrightarrow [H^1_{\mathcal I}(B^{(d-1)})]_{\bf n} \longrightarrow [H^1_{\mathcal I}(N_{d-1})]_{{\bf n}-{\bf e}} 
 \longrightarrow [H^1_{\mathcal I}(K^{(d-1)})]_{\bf n} \longrightarrow 
 [H^2_{\mathcal I}(B^{(d-1)})]_{\bf n} \longrightarrow \cdots
\end{eqnarray*}
is exact. Hence for ${\bf n} \geq d{\bf e}$ we have,
\begin{eqnarray} \label{isomorphism of N(d-1) and K(d-1)}
 [H^1_{\mathcal I}(N_{d-1})]_{{\bf n}-{\bf e}} \simeq [H^1_{\mathcal I}(K^{(d-1)})]_{\bf n}.
\end{eqnarray}
By the exact sequence (\ref{exact sequence relating N(d-1) and N(d)}), we also have the 
exact sequence 
\begin{eqnarray} \label{exact sequence relating N(d-1) and K(d-1)}
\cdots\longrightarrow[H^0_{\mathcal I}(N_d)]_{\bf n} \longrightarrow [H^1_{\mathcal I}(K^{(d-1)})]_{\bf n} 
\longrightarrow [H^1_{\mathcal I}(N_{d-1})]_{\bf n} \longrightarrow \cdots. 
\end{eqnarray}
Fix $(d-1){\bf e} \leq {\bf m} \in \mathbb Z^d $. By 
Corollary \ref{vanishing of lc of Nj for large n}, there exist an integer $t_0 \in \mathbb N$ such that 
$$[H^1_{\mathcal I}(N_{d-1})]_{{\bf m}+t{\bf e}}=0 \mbox{ for } t \geq t_0+1.$$
Therefore $\lm_R([H^1_{\mathcal I}(N_{d-1})]_{{\bf m}+t{\bf e}})<\infty$ 
for $t \gg 0$. We use decreasing induction on $t$ to show that 
$\lm_R([H^1_{\mathcal I}(N_{d-1})]_{{\bf m}+t{\bf e}}) < \infty$ for all $t \geq 0$. 
Assume the result for $t \geq1$ and 
prove it for $t-1$. Since  
$$[H^0_{\mathcal I}(N_d)]_{\bf n} \subseteq [H^0_{\mathcal R_{++}}(N_d)]_{\bf n} \subseteq [N_d]_{\bf n}
=\frac{\mathcal F({\bf n})}{(y_1,\ldots,y_d)\mathcal F({\bf n}-{\bf e})}\mbox{ for all }{\bf n} 
\in \mathbb Z^d,$$
$\lm_R([H^0_{\mathcal I}(N_d)]_{\bf n})< \infty$ for all ${\bf n}\in \mathbb Z^d$. 
Hence using induction hypothesis and 
the exact sequence (\ref{exact sequence relating N(d-1) and K(d-1)}), 
we obtain 
$$\lm_R([H^1_{\mathcal I}(K^{(d-1)})]_{{\bf m}+t{\bf e}})< \infty.$$
Since ${\bf m}+t{\bf e} \geq d{\bf e}$, from (\ref{isomorphism of N(d-1) and K(d-1)}), 
we get $\lm_R([H^1_{\mathcal I}(N_{d-1})]_{{\bf m}+(t-1){\bf e}})<\infty$. 
Thus $\lm_R([H^1_{\mathcal I}(N_{d-1})]_{{\bf m}+t{\bf e}})<\infty$ for all $t \geq 0$ and hence 
$\lm_R([H^1_{\mathcal I}(N_{d-1})]_{{\bf m}})<\infty$. Therefore 
$\lm_R([H^1_{\mathcal I}(N_{d-1})]_{{\bf m}})$ has finite length for all ${\bf m} \geq (d-1){\bf e}$. \\
Assume the result for $N_j$ for $j \geq 2$ and we prove it for $N_{j-1}$. Consider 
the exact sequence
 \begin{equation}\label{exact sequence relating N(j-1) and N(j)}
\diagram
 0\rto& B^{(j-1)}\rto&N_{j-1}(-{\bf e})\drto \rrto^{\theta_{j}}&& N_{j-1}\rto& N_j \rto &0\\
  &&&K^{(j-1)}\drto\urto&&&\\
 &&0\urto&&0&&
\enddiagram
\end{equation}
Note that 
\begin{eqnarray*}
 [N_{j-1}]_{\bf n}=\frac{\mathcal F({\bf n})}{(y_1,\ldots,y_{j-1})\mathcal F({\bf n}-{\bf e})}.
\end{eqnarray*}
Hence we have the sequence 
\begin{eqnarray*}
 0 \longrightarrow [B^{(j-1)}]_{\bf n} \longrightarrow \frac{\mathcal F({\bf n}-{\bf e})}
 {\sum_{i=1}^{j-1}y_i\mathcal F({\bf n}-2{\bf e})} \buildrel y_j \over \longrightarrow 
 \frac{\mathcal F({\bf n})}{\sum_{i=1}^{j-1}y_i \mathcal F({\bf n}-{\bf e})} \longrightarrow 0. 
\end{eqnarray*}
Let $\ov{p} \in [B^{(j-1)}]_{\bf n}$ for some $p \in \mathcal F({\bf n}-{\bf e})$ and 
$py_j \in \sum_{i=1}^{j-1}y_i \mathcal F({\bf n}-{\bf e})$. Since $y_1,\ldots,y_j$ is a 
regular sequence, we get 
$$p \in (y_1,\ldots,y_{j-1}) \cap \mathcal F({\bf n}-{\bf e})=
(y_1,\ldots,y_{j-1})
\mathcal F({\bf n}-2{\bf e}) \mbox{ for }{\bf n} -{\bf e}\geq (j-1){\bf e}, \mbox{ i.e, for } 
{\bf n} \geq j{\bf e}.$$
Therefore for ${\bf n} \geq j{\bf e}$, $[B^{(j-1)}]_{\bf n}=0$ which implies that 
$[H^i_{\mathcal I}(B^{(j-1)})]_{\bf n}=0$ for all ${\bf n}\geq j{\bf e}$ and all $i \geq 0$, by Lemma 
\ref{preliminary lemma}. By the exact sequence (\ref{exact sequence relating N(j-1) and N(j)}), 
we get 
\begin{eqnarray*}
 \cdots \longrightarrow [H^{d-j+1}_{\mathcal{I}}(B^{(j-1)})]_{\bf n} \longrightarrow [H^{d-j+1}_{\mathcal I}(N_{j-1})]_
 {{\bf n}-{\bf e}} \longrightarrow [H^{d-j+1}_{\mathcal I}(K^{(j-1)})]_{\bf n} 
 \longrightarrow [H^{d-j+2}_{\mathcal I}(B^{(j-1)})]_{\bf n} \longrightarrow \cdots
\end{eqnarray*}
is exact. Hence for all ${\bf n} \geq j{\bf e}$,
\begin{eqnarray} \label{isomorphism of N(j-1) and K(j-1)}
[H^{d-j+1}_{\mathcal I}(N_{j-1})]_{{\bf n}-{\bf e}} \simeq [H^{d-j+1}_{\mathcal I}
(K^{(j-1)})]_{\bf n}.
\end{eqnarray}
By the exact sequence (\ref{exact sequence relating N(j-1) and N(j)}), we obtain 
\begin{eqnarray}\label{exact sequence relating N(j-1) and K(j-1)}
\cdots \longrightarrow [H^{d-j}_{\mathcal I}(N_j)]_{\bf n} \longrightarrow [H^{d-j+1}_{\mathcal I}(K^{(j-1)})]_{\bf n} 
\longrightarrow [H^{d-j+1}_{\mathcal I}(N_{j-1})]_{\bf n} \longrightarrow \cdots
\end{eqnarray}
is exact. By induction hypothesis $[H^{d-j}_{\mathcal I}(N_j)]_{\bf n}$ 
has finite length for all ${\bf n} \geq j{\bf e}$. Fix $(j-1){\bf e} \leq {\bf m} \in \mathbb Z^d$. 
We use decreasing induction on 
$t$ to show that $\lm_R([H^{d-j+1}_{\mathcal I}(N_{j-1})]_{{\bf m}+t{\bf e}})< \infty$ 
for all $t \geq 0$. By Corollary \ref{vanishing of lc of Nj for large n}, 
the result is true for $t \gg 0$. 
Assume the result for $t\geq 1$ and prove it for $t-1$. 
As ${\bf m}+t{\bf e} \geq j{\bf e}$, by the exact sequence (\ref{exact sequence relating N(j-1) and K(j-1)}), $[H^{d-j+1}_{\mathcal I}(K^{(j-1)})]_{{\bf m}
+t{\bf e}}$ has finite length. Hence, from (\ref{isomorphism of N(j-1) and K(j-1)}), 
$[H^{d-j+1}_\mathcal I(N_{j-1})]_{{\bf m}+(t-1){\bf e}}$ has finite length. Thus 
$\lm_R([H^{d-j+1}_\mathcal I(N_{j-1})]_{{\bf m}+t{\bf e}})<\infty$ for all $t \geq 0$ and hence 
$\lm_R([H^{d-j+1}_\mathcal I(N_{j-1})]_{{\bf m}})<\infty$. Therefore 
$[H^{d-j+1}_\mathcal I(N_{j-1})]_{{\bf n}}$ has finite length for all ${\bf n} \geq (j-1){\bf e}.$
\end{proof}

In the following theorem we prove that for a strict complete reduction $(y_1,\ldots,y_d)$ of 
$\mathcal F$, $[H^d_{\mathcal I}(\mathcal R^\prime(\mathcal F))]_{\bf n}$ has finite length 
for all ${\bf n} \geq {\bf 0}$.

\begin{theorem}\label{finiteness of local cohomology modules}
 Let $(R,\m)$ be a $d-$dimensional Cohen-Macaulay local ring. Suppose ${\bf I}=(I_1,\ldots,I_d)$ 
 is a set of $\m-$primary ideals and $\mathcal F$ is an admissible ${\bf I}-$filtration. Suppose 
 $(x_{ij})$ is a strict complete reduction of $\mathcal F$ and $\mathcal I=(a_1t_1,\ldots,
 a_dt_d)$ be as before. Let $M=\mathcal R^\prime(\mathcal F)$. Then 
 \begin{eqnarray*}
  \lm_R([H^d_{\mathcal I}(M)]_{\bf n})<\infty \mbox{ for all }{\bf n} \geq {\bf 0}.
 \end{eqnarray*}
 \end{theorem}
\begin{proof}
 By Lemma \ref{vanishing of lc of Nj}, $[H^{d-1}_{\mathcal I}(N_1)]_{\bf n}$ 
 has finite length for all ${\bf n} \geq {\bf e}$. Fix ${\bf m} \in \mathbb N^d$. 
 Then ${\bf m}+t{\bf e} \gg {\bf 0}$ for all large positive integers $t.$ Hence, by 
 Corollary \ref{vanishing of lc for large n}, there 
 is an integer $t_0 \in \mathbb N$ such that 
 \begin{eqnarray} \label{finiteness of length for large n}
  [H^d_{\mathcal I}(M)]_{{\bf m}+t{\bf e}}=0 \mbox{ for }t \geq t_0+1.  
 \end{eqnarray}
 We use descending induction on $t$ to show that $\lm_R([H^d_{\mathcal I}(M)]_{{\bf m}+t{\bf e}}) < 
 \infty.$ From (\ref{finiteness of length for large n}), 
 $\lm_R([H^d_{\mathcal I}(M)]_{{\bf m}+t{\bf e}}) < \infty$ for $t \geq t_0+1$. 
 Assume the result for $t \geq 1$ and prove it for $t-1$.
 The exact sequence
\begin{eqnarray*}
 0 \longrightarrow M(-{\bf e})\buildrel{\theta_1}\over\longrightarrow M 
 \longrightarrow N_1 \longrightarrow 0
\end{eqnarray*}
gives the long exact sequence
\begin{eqnarray*}
 \cdots\longrightarrow [H^i_{\mathcal I}(M)]_{{\bf m}+(t-1){\bf e}} \longrightarrow 
 [H^i_{\mathcal I}(M)]_{{\bf m}+t{\bf e}} \longrightarrow [H^i_{\mathcal I}(N_1)]_{{\bf m}+t{\bf e}}
 \longrightarrow \cdots.
\end{eqnarray*}
Note that $[H^{d-1}_{\mathcal I}(N_1)]_{{\bf m}+t{\bf e}}$ 
 has finite length. Therefore using the exact sequence 
\begin{eqnarray*}
 \cdots \longrightarrow[H^{d-1}_{\mathcal I}(N_1)]_{{\bf m}+t{\bf e}} \longrightarrow 
 [H^d_{\mathcal I}(M)]_{{\bf m}+(t-1){\bf e}} \longrightarrow [H^d_{\mathcal I}(M)]_{{\bf m}+t{\bf e}} 
 \longrightarrow \cdots
\end{eqnarray*}
and induction hypothesis we obtain that $\lm_R([H^d_{\mathcal I}(M)]_{{\bf m}+
(t-1){\bf e}})<\infty$. Hence $\lm_R([H^d_{\mathcal I}(M)]_{{\bf m}+t{\bf e}})<\infty$ 
for all $t\geq 0$ and thus $\lm_R([H^d_{\mathcal I}(M)]_{\bf m})<\infty$. 
\end{proof}

We give an example to show that $[H^2_{(at_1,bt_2)}(\mathcal R^\prime(I,J))]_{(0,0)}$ need not have 
finite length in a Cohen-Macaulay local ring of dimension $2$ with infinite residue 
field and a joint reduction $(a,b)$ of $(I,J)$, where $\mathcal R^\prime(I,J)$ is the bigraded 
Rees algebra of the filtration $\{I^rJ^s\}$. 

\begin{example}\cite[Example 4.10]{dcruz-masuti}
 Let $R=k[|x,y|], \m=(x,y), I=(x^2,y^2), J=\m^2$. Let $a=x^2,b=y^2$. 
 Then, by \cite[Example 4.10]{dcruz-masuti}, $\lm_R([H^2_{(at_1,bt_2)}(\mathcal R^\prime(I,J))]_{(0,0)})$ 
 is not finite. We claim that there does not exists a strict complete reduction of $\{I^rJ^s\}$. 
 Suppose there exists a strict complete reduction $(x_{ij})$ of $\{I^rJ^s\}$. Then, 
 by Theorem \ref{finiteness of local cohomology modules}, 
 $\lm_R([H^2_{(x_{11}t_1,x_{22}t_2)}(\mathcal R^\prime(I,J))]_{(0,0)})< \infty$. 
 By \cite[Theorem 4.6]{dcruz-masuti}, it follows that 
 $\lm_R([H^2_{(at_1,bt_2)}(\mathcal R^\prime(I,J))]_{(0,0)})<\infty$, a contradiction. 
 \end{example}

\section{Existence of Good Complete Reductions and Good Joint Reductions}
\noindent Let $(R,\m)$ be a Noetherian local ring of dimension $3$ with infinite residue field and 
${\bf I}=(I_1,I_2,I_3)$ be a set of $\m-$primary ideals in $R$. 
In this section we prove existence of good complete reductions of an ${\bf I}-$admissible 
filtration $\mathcal F$ under some 
additional hypothesis on associated graded rings.  The following lemma is an analogue of 
\cite[Lemma 1.2]{rees3} for filtrations of ideals. Since the proof is similar, we skip it.
\begin{lemma} \em[Rees' Lemma]\label{r1.2}
 Let $(R,\m)$ be a Noetherian local ring of dimension $d$ with
infinite residue field and $I_1,\ldots,I_g$ be ideals in $R.$ 
Let $\mathcal F=\{\mathcal F({\bf n})\}$ be an $\bf I-$admissible filtration. 
Let $S$ be a finite set of prime ideals of $R$ not containing $I_1\ldots I_g.$ Then there
exist $x_i \in I_i$ not contained in any of 
the prime ideals in $S$ and integers $r_i$ such that 
\begin{eqnarray}\label{eq4}
(x_i) \cap \mathcal F({\bf n})=x_i \mathcal F({{\bf n}-{\bf e}_i}) 
\mbox{ for all } {\bf n} \geq r_i{\bf e}_i.
\end{eqnarray}
\end{lemma}
\noindent We put $u_i=t_i^{-1}$ and $u=u_1\ldots u_g$. The proof of 
\cite[Lemma 1.2]{rees3} shows the following:

\begin{lemma}\label{keylemma}
Let $\mathcal F$ be an $\bf I-$admissible filtration. Let $1 \leq i \leq g$ be fixed and 
\begin{eqnarray*}
 S_i=\{\p\in \ass_{\mathcal R^\prime(\mathcal F)} (\mathcal R^\prime(\mathcal F)/u\mathcal R^\prime(\mathcal F))\mid I_it_i \nsubseteq \p\}.
\end{eqnarray*}

\noindent Suppose there exists an element $x_i \in I_i $ such that $x_it_i \notin \p $ for all 
$\p \in S_i. $
Then there exists an integer $r_i$ such that
\begin{eqnarray*}
 (x_i) \cap \mathcal F({\bf n})&=&x_i \mathcal F({{\bf n}-{\bf e}_i})\mbox{ for all } {\bf n} \geq r_i{\bf e}_i. 
\end{eqnarray*}
 \end{lemma}

\begin{lemma}\label{lemma4}
Let $(R,\m)$ be a Noetherian local ring of dimension $d$ with infinite residue field  and $I_1,\ldots,I_g$ be ideals of positive height. 
 Let $\mathcal F$ be an $\bf I-$admissible filtration. For fixed $j$, $1 \leq j \leq g$, 
 suppose there exists a nonzerodivisor $x_j \in I_j \setminus \m I_j$ such that
\begin{eqnarray*}
 (x_j) \cap \mathcal F({{\bf n}})=x_j\mathcal F({\bf n}-{\bf e}_j) \mbox{ for }n_j 
 \gg 0 \mbox{ and for all }n_k \geq 0,k\neq j.
\end{eqnarray*}
Let $H^0_{\mathcal R_{++}}(G_i(\mathcal F))=0$ for $i=1,\ldots,g$.  
Then 
\begin{eqnarray*}
 (x_j) \cap \mathcal F({{\bf n}})=x_j\mathcal F({{\bf n}-{\bf e}_j}) \mbox{ for all }{\bf n}\geq {\bf e}_j 
\end{eqnarray*}
\end{lemma}
\begin{proof}
Since proof is similar for each $j$, we prove the result for $j=1$. 
Let $a=x_1$. Let $a^*$ denote the image of $a$ 
in $G_i(\mathcal F)_{{\bf e}_1}$. 
First we prove that $(0:_{G_i(\mathcal F)}a^*)=0$ for each $i=1,\ldots,g$. It 
suffices to prove the result for $G_1(\mathcal F)$. 
For $r \in \F({\bf n})$ let $r^*$ denote the image of $r$ in $[G_1(\F)]_{\bf n}$. 
Suppose $z^*a^*=0$ for some $z \in \mathcal F({\bf n})$, 
${\bf n} \geq {\bf 0}$. 
Then for any 
$y^*\in [G_1(\mathcal F)]_{\bf m}$, where $y \in \mathcal F({\bf m})$, 
$y^*z^*a^*=0$. Hence $yza \in (a) \cap \mathcal F({{\bf m}+{\bf n}+2{\bf e}_1})=
a \mathcal F({{\bf m}+{\bf n}+{\bf e}_1})$ for $m_1 \gg 0$. Thus $y^*z^*=0$ if $m_1 \gg 0$. 
Hence $\mathcal R_{++}^Nz^*=0$ for some $N$. Therefore 
$z^* \in [H^0_{\mathcal R_{++}}(G_1(\mathcal F))]_{\bf n}.$ Hence $z^*=0$. \\
Let $za\in (a) \cap \mathcal F({\bf n})$ for $n_1>0$. First we prove that 
$z\in \mathcal F({n_2{\bf e}_2}).$ Since $\mathcal F$ is an $\bf I-$admissible filtration, 
there exists an integer $ h \geq 0$ such that $\mathcal F({\bf n}) \subseteq {\bf I}^{{\bf n}-h{\bf e}}$ 
for all ${\bf n} \in \mathbb Z^g$. Hence 
$\cap_{k \geq0} \mathcal F({k{\bf e}_2}) \subseteq \cap_{k \geq h} I_2^{k-h} =0$. Thus there 
exists an integer $l \geq 0,$ such that $z \in \mathcal F({l{\bf e}_2})\setminus\mathcal F({(l+1){\bf e}_2})$. 
Suppose $l<n_2$. Then $z^* \in [G_2(\mathcal F)]_{l{\bf e_2}}$ and $z^*a^*=0$ as 
$za \in \mathcal F({\bf n})\subseteq \mathcal F({{\bf e}_1+(l+1){\bf e}_2})$, a contradiction. Therefore 
$z \in \mathcal F({n_2{\bf e}_2})$. \\
Similar argument shows that $z \in \mathcal F({n_2{\bf e}_2+\cdots+n_g{\bf e}_g})$. 
Since 
$$\cap_{k \geq 0} \mathcal F({k{\bf e}_1+n_2{\bf e}_2+\cdots+n_g{\bf e}_g}) \subseteq 
\cap_{k \geq 0} I_1^{k-h}=0,
$$
there exists an integer $l \geq 0$ such that 
$$z \in \mathcal F({l{\bf e}_1+n_2{\bf e}_2+\cdots+n_g{\bf e}_g})\setminus\mathcal F({(l+1){\bf e}_1+n_2{\bf e}_2+
\cdots+n_g{\bf e}_g}).$$
Suppose $l<n_1-1$. Then 
$z^*\in [G_1(\mathcal F)]_{l{\bf e}_1+n_2{\bf e}_2+\cdots+n_g{\bf e}_g}$ and $z^*a^*=0$, a contradiction. 
Hence $z \in \mathcal F({{\bf n}-{\bf e}_1})$. 
\end{proof}

\begin{theorem} \label{existence of good joint reductions}
 Let $(R,\m)$ be a Cohen-Macaulay local ring of 
dimension $3$ with infinite residue field $k$ and $I_1,I_2,I_3$ be $\m-$primary ideals in $R$. 
Let $\mathcal F$ be an $(I_1,I_2,I_3)-$admissible filtration. Assume that for $i=1,2,3$,
\begin{eqnarray*}
H^0_{\R_{++}}(G_i(\mathcal F))=0 \mbox{ and }
[H^1_{\R_{++}}(G_i(\mathcal F))]_{{\bf n}}=0\mbox{ for }
{\bf n}+{\bf e}_i \geq 0.
\end{eqnarray*}
Then there exists a good complete reduction $(M_{ij})$, where
$M= \begin{bmatrix}
  x_1 & x_2 & x_3\\
  y_1 & y_2 & y_3\\
  z_1 & z_2 & z_3\\
 \end{bmatrix}
$ 
such that for each $A \subsetneq \{1,2,3\}$ 
\begin{eqnarray*}
 (x_i:i \in A) \cap \mathcal F{(r,s,t)}&=&(x_i:i\in A)\mathcal F{(r-1,s,t)} \mbox{ for }r >0 \mbox{ and all } s,t \geq 0 \\
 (y_i:i \in A) \cap \mathcal F{(r,s,t)}&=&(y_i:i\in A)\mathcal F{(r,s-1,t)}) \mbox{ for }s >0 \mbox{ and all } r,t \geq 0 \\
 (z_i:i \in A) \cap \mathcal F{(r,s,t)}&=&(z_i:i\in A)\mathcal F{(r,s,t-1)} \mbox{ for }t >0 \mbox{ and all } r,s \geq 0.
\end{eqnarray*}
In particular, $(M_{i,\sigma (i)})$ is a good joint reduction of $\mathcal F$ for each 
permutation $\sigma \in S_3$. 
\end{theorem}
\begin{proof}
From Lemma \ref{r1.2} and \ref{lemma4}, there exist nonzerodivisors $x_1\in I_1 \setminus \m I_1,
y_1\in I_2 \setminus \m I_2,z_1\in I_3 \setminus \m I_3$ such that 
\begin{eqnarray*}
 (x_1) \cap \mathcal F{(r,s,t)}&=&x_1 \mathcal F{(r-1,s,t)}\mbox{ for }r>0 \mbox{ and all }s,t\geq 0\\
(y_1) \cap \mathcal F{(r,s,t)}&=&y_1 \mathcal F{(r,s-1,t)}\mbox{ for }s >0 \mbox{ and all }r,t \geq 0\\
(z_1) \cap \mathcal F{(r,s,t)}&=&z_1 \mathcal F{(r,s,t-1)}\mbox{ for }t>0 \mbox{ and all }r,s \geq 0.
\end{eqnarray*}
Let $a=x_1y_1z_1$, $S_1=R/(x_1),S_2=R/(y_1),S_3=R/(z_1)$ and $S=R/(a)$. 
Consider the filtration $\mathcal G=\{\mathcal F{(r,s,t)}S\}$.  
Let 
\begin{eqnarray*}
A_i&=&\{\p \in \ass_{\mathcal R^\prime(\mathcal G)}(\mathcal R^\prime(\mathcal G)/
(u\mathcal R^\prime(\mathcal G))) \mid I_iSt_i \notin \p\}\\
B_i&=&\{\p \in \ass_{\mathcal R^\prime(\mathcal F)}(\mathcal R^\prime(\mathcal F)/
  (u\mathcal R^\prime(\mathcal F)))\mid I_it_i \notin \p\} \mbox{ and }\\
C_{i,j}&=&\{\p \in \ass_{\mathcal R^\prime(\mathcal FS_j)}(\mathcal R^\prime(\mathcal FS_j)/
(u \mathcal R^\prime(\mathcal FS_j))) \mid I_iS_jt_i \notin \p \}.
\end{eqnarray*}
Then $\frac{\p \cap I_1t_1+ \m I_1t_1 }{\m I_1t_1}$
(respectively $W_j=\frac{\p \cap I_1S_jt_1+\m I_1S_jt_1}{\m I_1S_jt_1}$ and 
$W=\frac{\p \cap I_1St_1+\m I_1St_1}{\m I_1St_1}$) is a 
proper subspace of $\frac{I_1t_1}{\m I_1t_1} \cong \frac{I_1}{\m I_1}$ 
(respectively $\frac{I_1S_jt_1}{\m I_1S_jt_1}$ and $\frac{I_1St_1}{\m I_1St_1}$) 
for $\p \in B_1$ (respectively $\p \in C_{1,j}$ and $\p \in A_1$). 
Let 
$$f_j: \frac{I_1}{\m I_1} \longrightarrow \frac{I_1S_j}{\m I_1S_j} \mbox{ and }
f:\frac{I_1}{\m I_1} \longrightarrow \frac{I_1S}{\m I_1S}$$ be natural maps of $k-$vector spaces. 
Since $f_j$ and $f$ are surjective, $f_j^{-1}(W_j) $ and $f^{-1}(W)$ are proper subspaces of 
$\frac{I_1}{\m I_1}$. Since $k$ is infinite, there exists an element $x_2 \in I_1$ such that 
$ x_2t_1 \notin \p$ for any $\p \in B_1 $, $x_2^\prime t_1\notin \p$ for any $\p \in  C_{1,j}, 
A_1$ and $x_2$ is $S-$regular. Here $^\prime$ denotes the image of an element in respective quotients.  
Similarly, there exist $y_2\in I_2
$ and $z_2 \in I_3$ such that $y_2t_2 \notin \p$ for any $\p \in B_2,y_2^\prime t_2 \notin \p$ 
for any $\p \in A_2,C_{2,j}$, 
$z_2t_3 \notin \p$ for any $\p \in B_3,z_2^\prime t_3 \notin \p$ for any $\p\in A_3,C_{3,j}$ and 
$y_2,z_2 $ are nonzerodivisors in $S$. Therefore by Lemma 
\ref{keylemma} and \ref{lemma4}
\begin{eqnarray*}
 (x_2) \cap \mathcal F{(r,s,t)}&=&x_2 \mathcal F{(r-1,s,t)} \mbox{ for }r>0 \mbox{ and all } s,t \geq0\\
 x_2S \cap \mathcal G{(r,s,t)}&=&x_2\mathcal G{(r-1,s,t)}\mbox{ for }r \gg 0 \mbox{ and all }s,t \geq 0\\
x_2S_j \cap \mathcal F{(r,s,t)}S_{j}&=&x_2\mathcal F{(r-1,s,t)}S_j \mbox{ for } r \gg 0 \mbox{ and all }s,t \geq 0. 
 \end{eqnarray*}
\noindent Let $y_1^*$ denote the image of $y_1$ in $[G_i(\mathcal F)]_{{\bf e}_2}$. 
Note that $y_1^{*}$ is $G_i(\mathcal F)-$regular for each $i=1,2,3$. 
Therefore we have an exact sequence
$$0 \longrightarrow G_2(\mathcal F)(0,-1,0)\buildrel_{\mu_{y_1^*}} 
\over\longrightarrow G_2(\mathcal F) \longrightarrow 
G_2(\mathcal F)/(y_1^*) \longrightarrow 0.$$
This gives the long exact sequence 
$$\cdots\longrightarrow H^0_{\R_{++}}(G_2(\mathcal F)) 
\longrightarrow H^0_{\R_{++}}(G_2(\mathcal F)/(y_1^*)) 
\longrightarrow H^1_{\R_{++}}(G_2(\mathcal F)(0,-1,0)) \longrightarrow \cdots.$$ 
Therefore $[H^0_{\R_{++}}(G_2(\mathcal F)/(y_1^*))]_{(r,s,t)}=0$ 
for all $r,s,t \geq 0$. 
Consider the exact sequence
$$0 \longrightarrow L \longrightarrow G_2(\mathcal F)/(y_1^*) 
\longrightarrow G_2(\mathcal FS_2) \longrightarrow 0 .$$
Since $(y_1) \cap \mathcal F{(r,s,t)}=y_1\mathcal F{(r,s-1,t)}$ 
for $s>0$, $[G_2(\mathcal F)/(y_1^*)]_{(r,s,t)}=[G_2(\mathcal FS_2)]_{(r,s,t)}$ 
for $s>0$. Thus $L_{(r,s,t)}=0$ for $s>0$ and hence $L$ is $\R_{++}-$torsion. Hence the sequence 
$$0 \longrightarrow L \longrightarrow H^0_{\R_{++}}(G_2(\mathcal F)/(y_1^*))
\longrightarrow H^0_{\R_{++}}(G_2(\mathcal FS_2)) \longrightarrow 0 $$
is exact. Thus 
$$H^0_{\R_{++}}(G_2( \mathcal FS_2))=0 .$$ 
Considering an exact sequence
$$0 \longrightarrow G_1(\mathcal F)(0,-1,0)\buildrel_{\mu_{y_1^*}} 
\over\longrightarrow G_1(\mathcal F) \longrightarrow 
G_1(\mathcal F)/(y_1^*) \longrightarrow 0.$$
we get the long exact sequence 
$$\cdots\longrightarrow H^0_{\R_{++}}(G_1(\mathcal F)) 
\longrightarrow H^0_{\R_{++}}(G_1(\mathcal F)/(y_1^*)) 
\longrightarrow H^1_{\R_{++}}(G_1(\mathcal F)(0,-1,0)) \longrightarrow \cdots.$$
This gives $[H^0_{\R_{++}}(G_1(\mathcal F)/(y_1^*))]_{(r,s,t)}=0$ 
for $s \geq 1$ and $r,t \geq 0$. Consider the exact sequence
$$0 \longrightarrow L^\prime \longrightarrow G_1(\mathcal F)/(y_1^*) 
\longrightarrow G_1(\mathcal FS_2) \longrightarrow 0 .$$
Since $(y_1) \cap \mathcal F{(r,s,t)}=y_1\mathcal F{(r,s-1,t)}$ 
for $s>0$, $[G_1(\mathcal F)/(y_1^*)]_{(r,s,t)}=[G_1(\mathcal FS_2)]_{(r,s,t)}$ 
for $s>0$. Thus $L^\prime_{(r,s,t)}=0$ for $s>0$ and hence $L^\prime$ is $\R_{++}-$torsion. Hence the sequence 
$$0 \longrightarrow L^\prime \longrightarrow H^0_{\R_{++}}(G_1(\mathcal F)/(y_1^*))
\longrightarrow H^0_{\R_{++}}(G_1(\mathcal FS_2)) \longrightarrow 0 $$
is exact. Thus 
$$[H^0_{\R_{++}}(G_1( \mathcal FS_2))]_{(r,s,t)}=0 \mbox{ for }s \geq 1.$$ 
Similar argument shows that  
$$[H^0_{\R_{++}}(G_3( \mathcal FS_2))]_{(r,s,t)}=0 
\mbox{ for }s \geq1.$$
We prove that for all $r,s \geq 1$,
$$(y_1,x_2)\cap \mathcal F{(r,s,t)}=y_1\mathcal F{(r,s-1,t)}+x_2\mathcal F{(r-1,s,t)}.$$
First we claim that 
\begin{eqnarray*}
(0:_{G_2(\mathcal FS_2)}(x_2^\prime)^*)=0 \mbox{ and } 
[0:_{G_i(\mathcal FS_2)}(x_2^\prime)^*]_{(r,s,t)}=0 \mbox{ if }s \geq 1 \mbox{ for } i=1,3.
\end{eqnarray*} 
Here $^\prime$ denotes the image of an element in $S_2$. 
Suppose $(z^\prime)^*(x_2^\prime)^*=0$ for some $(z^\prime)^* \in 
[G_2(\mathcal FS_2)]_{(r,s,t)} $. Then $(y^\prime)^*(z^\prime)^*(x_2^\prime)^*=0$ 
for any $(y^\prime)^* \in [G_2(\mathcal FS_2)]_{(k,l,m)}$. Hence 
$$(yzx_2)^\prime \in \mathcal F{(r+k+1,s+l+1,t+m)}S_2 \mbox{ and } (yz)^\prime \in 
\mathcal F{(r+k,s+l+1,t+m)}S_2$$ 
for $k \gg 0$. Thus $(y^\prime)^*(z^\prime)^*=0$ in 
$G_2(\mathcal FS_2)$ for $k \gg 0$. Hence $\R_{++}^N(z^\prime)^*=0$ for 
some $N$ large. Therefore $(z^\prime)^*\in [H^0_{\mathcal R_{++}}(G_2(\mathcal FS_2))]_{(r,s,t)} =0$. \\
Suppose $(z^\prime)^*(x_2^\prime)^*=0$ for some $(z^\prime)^* \in 
[G_1(\mathcal FS_2)]_{(r,s,t)} $ and $s \geq 1$. Then $(y^\prime)^*(z^\prime)^*(x_2^\prime)^*=0$ 
for any $(y^\prime)^* \in [G_1(\mathcal FS_2)]_{(k,l,m)}$. Hence 
$(yzx_2)^\prime \in \mathcal F{(r+k+2,s+l,t+m)}S_2$. Hence $(yz)^\prime \in 
\mathcal F{(r+k+1,s+l,t+m)}S_2$ for $k \gg 0$. Thus $(y^\prime)^*(z^\prime)^*=0$ in 
$G_1(\mathcal FS_2)$ for $k \gg 0$. Hence $\R_{++}^N(z^\prime)^*=0$ for 
some $N$ large. Therefore $(z^\prime)^*\in [H^0_{\mathcal R_{++}}(G_1(\mathcal FS_2))]_{(r,s,t)} =0$. 
Similar argument shows that $[0:_{G_3(\mathcal FS_2)} (x_2^\prime)^*]_{(r,s,t)}=0$ 
for $s \geq 1$.\\
Let $y_1u+x_2v \in \F(r,s,t)$. Since $\cap_{n\geq 0}\mathcal F(0,n,0)S_2 \subseteq 
\cap_{n\geq 0} I_2^{n-h}S_2=0$, there exists an integer $l\geq 0$ such that 
$v^\prime \in \mathcal F{(0,l,0)}S_2\setminus\mathcal F{(0,l+1,0)}S_2$. 
Suppose $l<s$. 
Then $x_2^\prime v^\prime=(x_2v)^\prime \in \F(r,s,t)S_2 \subseteq \F(1,l+1,0)S_2$. Hence 
$(x_2^\prime)^* (v^\prime)^*=0$ in $G_2(\F S_2)$. 
Thus $0\neq(v^\prime)^*\in (0:_{G_2(\mathcal FS_2)}(x_2^\prime)^*)=0$, a contradiction. Therefore 
$l \geq s$. Hence $v^\prime \in \F(0,s,0)S_2=\frac{\F(0,s,0)+(y_1)}{(y_1)}$. Therefore $v=
v_1+y_1v_2$ for some $v_1 \in \F(0,s,0)$ and $v_2 \in R$. Thus 
$$y_1u+x_2v=y_1(u+x_2v_2)+x_2v_1 \in \F(r,s,t).$$
Suppose $v_1^\prime \in \F(l,s,0)S_2 \setminus \F(l+1,s,0)S_2$ and $l<r-1$. 
Since $v_1^\prime x_2^\prime \in \F(r,s,t)S_2 \subseteq \F(l+2,s,0)$, 
$(v_1^\prime)^*(x_2^\prime)^*=0$ in $G_1(\F S_2)$. Thus $0\neq (v_1^\prime)^* \in 
[0:_{G_1(\mathcal FS_2)} (x_2^\prime)^*]_{(r,s,t)} $ for $s \geq 1$, a contradiction.  
Therefore $l \geq r-1$. Hence $v_1^\prime \in \F(r-1,s,0)S_2$. Let $v_1=v_3+y_1v_4$ for some 
$v_3 \in \F(r-1,s,0) $ and $v_4 \in R$. Thus 
$$y_1u+x_2v=y_1(u+x_2v_2+x_2v_4)+x_2v_3 \in \F(r,s,t).$$
Suppose $v_3^\prime \in \F(r-1,s,l)S_2 \setminus \F(r-1,s,l+1)S_2$ for some $l<t$. 
Since $v_3^\prime x_2^\prime \in \F(r,s,t)S_2 \subseteq \F(r,s,l+1)S_2$, 
$(v_3^\prime)^*(x_2^\prime)^*=0$ in $G_3(\F S_2)$. Thus $0\neq (v_3^\prime)^* \in 
[0:_{G_3(\mathcal FS_2)} (x_2^\prime)^*]_{(r,s,t)} $ and $s \geq 1$, a contradiction.  
Therefore $l \geq t$. Let $v_3=v_5+y_1v_6$ for some $v_5 \in \F(r-1,s,t)$ and 
$v_6 \in R$. Then 
$$y_1u+x_2v=y_1(u+x_2v_2+x_2v_4+x_2v_6)+x_2v_5 \in\F(r,s,t).$$
Let $u_1=u+x_2v_2+x_2v_4+x_2v_6$. Since $x_2v_5 \in \F(r,s,t)$, 
$y_1u_1 \in \F(r,s,t)\cap (y_1)=y_1\F(r,s-1,t)$. Hence 
$u_1 \in \F(r,s-1,t)$. Thus 
$$y_1u+x_2v=y_1u_1+x_2v_5 \in y_1\mathcal F{(r,s-1,t)}+x_2\mathcal F{(r-1,s,t)}.$$
Similar argument shows that 
\begin{eqnarray*}
(y_1,z_2) \cap \mathcal F{(r,s,t)}&=&y_1\mathcal F{(r,s-1,t)}+z_2\mathcal F{(r,s,t-1)}\mbox{ for } 
s,t \geq 1 \mbox{ and all }r \geq 0 \mbox{ and }\\
(y_1,y_2) \cap \mathcal F{(r,s,t)}&=&(y_1,y_2)\mathcal F{(r,s-1,t)} \mbox{ for }s \geq 1 
\mbox{ and all }r,t \geq 0.
\end{eqnarray*}
Similarly, considering the ring $S_1$ and $S_3$, we get  
\begin{eqnarray*}
H^0_{\R_{++}}(G_1(\F S_1))&=&0 \mbox{ and } [H^0_{\R_{++}}(G_i(\mathcal FS_1))]_{(r,s,t)}=0 
\mbox{ for all }r \geq 1 \mbox{ and }i=2,3 \mbox{ and }\\
H^0_{\R_{++}}(G_3(\F S_3))&=&0 \mbox{ and }[H^0_{\R_{++}}(G_i(\mathcal FS_3))]_{(r,s,t)}=0 
\mbox{ for all }t \geq 1 \mbox{ and }i=1,2.
\end{eqnarray*} Thus 
\begin{eqnarray*}
 (x_1,y_2) \cap \mathcal F{(r,s,t)}&=&x_1\mathcal F{(r-1,s,t)}+y_2 \mathcal F{(r,s-1,t)} \mbox{ for }
 r,s \geq 1 \mbox{ and }t \geq 0\\
 (x_1,z_2) \cap \mathcal F{(r,s,t)}&=&x_1\mathcal F{(r-1,s,t)}+z_2\mathcal F{(r,s,t-1)} \mbox{ for }
r,t \geq 1 \mbox{ and }s \geq 0\\
(x_1,x_2)\cap \mathcal F{(r,s,t)}&=&(x_1,x_2)\mathcal F{(r-1,s,t)} \mbox{ for }r \geq 1 \mbox{ and all }s,t \geq 0\\
(z_1,x_2) \cap \mathcal F{(r,s,t)}&=&z_1\mathcal F{(r,s,t-1)}+x_2\mathcal F{(r-1,s,t)} \mbox{ for }
r,t \geq 1 \mbox{ and }s\geq 0\\
(z_1,y_2) \cap \mathcal F{(r,s,t)}&=&z_1\mathcal F{(r,s,t-1)}+y_2\mathcal F{(r,s-1,t)} \mbox{ for }
s,t \geq 1 \mbox{ and }r \geq 0\\
(z_1,z_2) \cap \mathcal F{(r,s,t)}&=&(z_1,z_2)\mathcal F{(r,s,t-1)} \mbox{ for }t\geq 1 \mbox{ and all }r,s \geq 0. 
\end{eqnarray*}
Let $b=x_2y_2z_2$. We claim that
$$(a,b) \cap \mathcal F{(r,s,t)}=(a,b)\mathcal F{(r-1,s-1,t-1)} \mbox{ for }r,s,t > 1.$$
First we prove $(a)\cap \mathcal F(r,s,t)=a\mathcal F(r-1,s-1,t-1)$ for $r,s,t>0$. Let 
$au \in \mathcal F(r,s,t). $ Then $x_1y_1z_1u \in  \F(r,s,t)$. Since $r>0$, $y_1z_1u \in \F(r-1,s,t)$. 
Since $s>0$, $z_1u \in \F(r-1,s-1,t)$. As $t>0$, $u \in \F(r-1,s-1,t-1).$\\
Next we prove that $(a,x_2)\cap \F(r,s,t)=a\F(r-1,s-1,t-1)+x_2\F(r-1,s,t)$ for $r,s,t>0$. 
Let $ap+x_2q\in \F(r,s,t)$. Hence $x_1y_1z_1 p+x_2q \in (y_1,x_2) \cap \F(r,s,t)$. Let 
$x_1y_1z_1 p+x_2q=y_1p_1+x_2 q_1$ for some $p_1 \in \F(r,s-1,t) $ and $q_1 \in \F(r-1,s,t)$. Thus 
$y_1(x_1z_1p-p_1)\in (x_2)$. Hence $x_1z_1p-p_1 \in (x_2)$. Therefore $p_1 \in (x_2,z_1) \cap \F(r,s-1,t)=
x_2\F(r-1,s-1,t)+z_1\F(r,s-1,t-1)$. Let $p_1=x_2p_2+z_1p_3$ for some $p_2 \in \F(r-1,s-1,t) $ 
and $p_3\in \F(r,s-1,t-1)$. Then 
$$ap+x_2q=y_1p_1+x_2 q_1=y_1z_1p_3+x_2(q_1+y_1p_2).$$
Therefore $y_1z_1(x_1p-p_3) \in (x_2).$ Thus $x_1p-p_3 \in (x_2)$. Hence $p_3 \in 
(x_1,x_2) \cap \F(r,s-1,t-1)$. Let $p_3=x_1p_4+x_2p_5$ for some $p_4,p_5 \in 
\F(r-1,s-1,t-1)$. Thus 
$$ap+x_2q=ap_4+x_2(q_1+y_1p_2+y_1z_1p_5) \in a\F(r-1,s-1,t-1)+x_2\F(r-1,s,t).$$
Similar argument shows that 
\begin{eqnarray*}
(a,y_2)\cap \F(r,s,t)&=&a\F(r-1,s-1,t-1)+y_2\F(r,s-1,t)\mbox{ for }r,s,t>0\mbox{ and }\\
(a,z_2)\cap \F(r,s,t)&=&a\F(r-1,s-1,t-1)+z_2\F(r,s,t-1)\mbox{ for }r,s,t>0. 
\end{eqnarray*}
Next we prove that 
$(a,b) \cap \mathcal F{(r,s,t)}=(a,b)\mathcal F{(r-1,s-1,t-1)} \mbox{ for }r,s,t > 1$. 
Let $ap+bq \in \F(r,s,t)$. Then $ap+bq \in (a,x_2) \cap \F(r,s,t)=
a\F(r-1,s-1,t-1)+x_2\F(r-1,s,t)$. Let $ap+bq=ap_1+x_2q_1$ for some $p_1\in \F(r-1,s-1,t-1)$ and 
$q_1 \in \F(r-1,s,t)$. Then $x_2(y_2z_2q-q_1) \in (a)$. Hence $y_2z_2q-q_1\in (a)$. Thus 
$q_1\in (a,y_2) \cap \F(r-1,s,t)=a\F(r-2,s-1,t-1)+y_2\F(r-1,s-1,t)$. Let 
$q_1=aq_2+y_2q_3$ for some $q_2 \in \F(r-2,s-1,t-1)$ and $q_3 \in \F(r-1,s-1,t)$. Thus 
$$ap+bq=a(p_1+x_2q_2)+x_2y_2q_3.$$
Hence $x_2y_2(z_2q-q_3)\in (a)$. Thus 
$$q_3 \in (a,z_2) \cap \F(r-1,s-1,t)=
a\F(r-2,s-2,t-1)+z_2\F(r-1,s-1,t-1).$$ 
Let $q_3=aq_4+z_2q_5$ for some $q_4 \in \F(r-2,s-2,t-1)$ 
and $q_5 \in \F(r-1,s-1,t-1)$. Hence 
$$ap+bq=a(p_1+x_2q_2+x_2y_2q_4)+bq_5 \in (a,b) \F(r-1,s-1,t-1).$$
\noindent Let $T_1=R/(x_2),T_2=R/(y_2),T_3=R/(z_2)$ and $T=R/(a,b)$. Consider the filtration 
$\mathcal H=\{\mathcal F{(r,s,t)}T\}.$ Let 
\begin{eqnarray*}
D_i&=&\{\p \in \ass_{\mathcal R^\prime(\mathcal H)}(\mathcal R^\prime(\mathcal H)/
  (u\mathcal R^\prime(\mathcal H)))\mid I_iTt_i \notin \p\} \mbox{ and }\\
E_{i,j}&=&\{\p \in \ass_{\mathcal R^\prime(\mathcal FT_j)}(\mathcal R^\prime(\mathcal FT_j)/
  (u\mathcal R^\prime(\mathcal FT_j)))\mid I_iT_jt_i \notin \p\}.
\end{eqnarray*}
  Then there exist $x_3 \in I_1,y_3 \in I_2$ and $z_3 \in I_3$ such that $x_3,y_3,z_3$ are nonzerodivisors in $T$ and 
$x_3t_1 \notin \p $ for any $\p \in B_1$, $x_3^\prime t_1 \notin \p$ for any $\p \in 
C_{1,j},D_1,E_{1,j},y_3t_2\notin \p$ for any 
$\p \in B_2$, $y_3^\prime t_2 \notin \p $ for any $\p \in C_{2,j},D_2,E_{2,j}$ and 
$z_3t_3 \notin \p $ for any $p\in B_3, z_3^\prime t_3 \notin \p$ for any $\p \in C_{3,j},D_3,E_{3,j}$. 
Therefore by Lemma \ref{keylemma} 
\begin{eqnarray*}
 (x_3) \cap \mathcal F{(r,s,t)}&=&x_3 \mathcal F{(r-1,s,t)} \mbox{ for }r>0 \mbox{ and all } s,t \geq0\\
 x_3S_j \cap \mathcal F{(r,s,t)}S_j&=&x_3\mathcal F{(r-1,s,t)}S_j\mbox{ for }r \gg 0 \mbox{ and all }s,t \geq 0\mbox{ and }\\
x_3T_j \cap \mathcal F{(r,s,t)}T_j&=&x_3 \mathcal F{(r-1,s,t)}T_j \mbox{ for }r\gg0 \mbox{ and all } s,t \geq0 \mbox{ and }\\
 x_3T \cap \mathcal H{(r,s,t)}&=&x_3 \mathcal H{(r-1,s,t)} \mbox{ for }r\gg0 \mbox{ and all } s,t \geq0.
\end{eqnarray*}
Then argument as above shows that for $i=1,2$
\begin{eqnarray*}
 (y_i,x_3) \cap \mathcal F{(r,s,t)}&=&y_i \mathcal F{(r,s-1,t)}+x_3\mathcal F{(r-1,s,t)} 
 \mbox{ for }r,s>0 \mbox{ and all }t \geq 0 \mbox{ and }\\
(z_i,x_3) \cap \mathcal F{(r,s,t)}&=&z_i \mathcal F{(r,s,t-1)}+x_3\mathcal F{(r-1,s,t)} 
 \mbox{ for }r,t>0 \mbox{ and all }s \geq 0 \\
 (x_1,x_3) \cap \mathcal F{(r,s,t)}&=&(x_1,x_3)\mathcal F{(r-1,s,t)} \mbox{ for }r \geq 1 \mbox{ and all }s,t \geq 0\\
 (x_2,x_3) \cap \mathcal F{(r,s,t)}&=&(x_2,x_3)\mathcal F{(r-1,s,t)} \mbox{ for }r \geq 1 \mbox{ and all }s,t \geq 0.
 \end{eqnarray*}
Similarly $y_3,z_3$ satisfies required equations. Let $c=x_3y_3z_3$. Then
$$(a,b,c) \cap \mathcal F{(r,s,t)}=(a,b,c) \mathcal F{(r-1,s-1,t-1)} \mbox{ for all }r,s,t \gg 0.$$
Since $R/(a,b,c)$ is Artinian, $\mathcal F{(r,s,t)} \subseteq (a,b,c)$ for all $r,s,t$ large. Thus
$$\mathcal F{(r,s,t)}=(a,b,c)\mathcal F{(r-1,s-1,t-1)} \mbox{ for }r,s,t \gg 0.$$
Hence $M= \begin{bmatrix}
  x_1 & x_2 & x_3\\
  y_1 & y_2 & y_3\\
  z_1 & z_2 & z_3\\
 \end{bmatrix}
$ is a good complete reduction of $\mathcal F$ with the desired property.
\end{proof} 

In what follows we prove that a good complete reduction of $\mathcal F$ exists if 
$\mathcal R(\mathcal F)$ is Cohen-Macaulay. We prove few lemmas required for this purpose. 
For a homomorphism $\phi:\mathbb Z^r \longrightarrow \mathbb Z^q$, set 
$$M^\phi=\bigoplus_{{\bf m} \in \mathbb Z^q}\left(\bigoplus_{\phi({\bf n})={\bf m}} M_{\bf n}\right)$$
for any $\mathbb Z^r-$graded module $M$ defined over $\mathbb Z^r-$graded ring $T$. Then 
$H^i_{\mathfrak a}(M)^\phi=H^i_{\mathfrak a^\phi}(M^\phi)$ for any homogeneous ideal 
$\mathfrak{a}$ in $T$. 
See \cite{hyry}.

\begin{lemma} \label{vanishing of lc T++}
 Let $T$ be standard $\mathbb N^3-$graded ring defined over a local ring $(R,\m)$ and $M$ be a finitely 
 generated $\mathbb N^3-$graded $T-$module. Let $\mathcal M$ be the maximal homogeneous ideal of $T$. 
 Suppose that each $j=1,2,3$ satisfies, 
 $$[H^i_\mathcal M(M)]_{{\bf n}}=0  \mbox{ for all }i\geq 0 \mbox{ and }n_j\geq 0 . $$
Then 
$$[H^i_{T_{++}}(M)]_{{\bf n}}=0 \mbox{ for all }{\bf n} \in \mathbb N^3 \mbox{ and all }i \geq 0.
$$
 \end{lemma}
\begin{proof}
Let $T_i^+=\bigoplus_{n_i \geq 1} T_{\bf n}$ and $T_{i,j}^+=\bigoplus_{n_i,n_j \geq 1}T_{\bf n}$. 
Consider a homomorphism $\phi: \mathbb Z^3 \longrightarrow \mathbb Z$ defined as 
$\phi(r,s,t)=r$. Then $M^\phi$ is an $\mathbb N-$graded module defined over the $\mathbb N-$graded 
ring $T^\phi$.
Let $B_{(r,s)}=T_{(0,r,s)}$ and $B=\bigoplus_{r,s \geq 0} B_{(r,s)}$. Then $\mathcal N=\m\bigoplus_{r+s \geq 1}B_{(r,s)}$ 
is the maximal homogeneous ideal of $B$. Then $T^\phi\otimes_B B_\mathcal{N}$ is $\mathbb N-$graded ring 
defined over a local ring $B_\mathcal N$. We have 
$[H^i_{\mathcal M^\phi}(M^{\phi})]_{m}=\bigoplus_{s,t \in \mathbb Z}[H^i_{\mathcal M}(M)]_{(m,s,t)}$. 
Hence $[H^i_{\mathcal M^\phi}(M^{\phi})]_{m}=0$ for all $m \geq 0$ and $i \geq 0$. Therefore
\begin{eqnarray*}
 [H^i_{\mathcal M^\phi \otimes_B B_\mathcal N}(M^\phi \otimes_B B_\mathcal N)]_{m} \cong 
 [H^i_{\mathcal M^\phi}(M^\phi)]_{m} \otimes_B B_{\mathcal N}=0 \mbox{ for all } m \geq 0 \mbox{ and }i \geq 0.
\end{eqnarray*}
Taking $\mathfrak a=0$ in \cite[Lemma 2.3]{hyry}, we get that 
$$[H^i_{(T_1^+)^\phi}(M^\phi)]_{m} \otimes_B B_{\mathcal N}\cong 
[H^i_{(T_1^+)^\phi \otimes_B B_\mathcal N}(M^\phi \otimes_B B_\mathcal N)]_{m}=0 
\mbox{ for all }m \geq 0 \mbox{ and }i \geq 0.$$
Since $[H^i_{(T_1^+)^\phi}(M^\phi)]_{m}$ is bigraded $B-$module and $([H^i_{(T_1^+)^\phi}(M^\phi)]_{m})_\mathcal N=0$, 
$[H^i_{(T_1^+)^\phi}(M^\phi)]_{m}=0$ for all $m \geq 0$. Since $[H^i_{(T_1^+)^\phi}(M^\phi)]_{m}=
\bigoplus_{s,t \geq 0}[H^i_{T_1^+}(M)]_{(m,s,t)}=0$, 
\begin{eqnarray}\label{vanishing of lc 1+}
[H^i_{T_1^+}(M)]_{\bf n}=0 \mbox{ for all }n_1 \geq 0 \mbox{ and }i \geq 0.
\end{eqnarray}
Taking $\mathfrak{a}=(\bigoplus_{p \geq 1,q \geq 0}{B_{(p,q)}})\otimes_B B_{\mathcal N} $ in 
\cite[Lemma 2.3]{hyry}, we get that 
\begin{eqnarray} \label{vanishing of lc 12+}
[H^i_{T_1^++T_2^+}(M)]_{\bf n}&=&0 \mbox{ for all }n_1 \geq 0 \mbox{ and all }i \geq 0.
\end{eqnarray}
Similar argument shows that 
\begin{eqnarray} \label{vanishing of lc i+}
[H^j_{T_i^+}(M)]_{\bf n}=0 \mbox{ for all }n_i \geq 0 \mbox{ and }j \geq 0. 
\end{eqnarray}
Consider $\phi:\mathbb Z^3 \longrightarrow \mathbb Z$ defined as $\phi(r,s,t)=t$. Let 
$C=\bigoplus_{r,s \geq 0}T_{(r,s,0)}$ and $\mathcal P$ be the maximal homogeneous ideal of $C$. Then 
taking $\mathfrak a=(\bigoplus_{p,q \geq 1}{C_{(p,q)}})\otimes_C C_{\mathcal P}$ in 
\cite[Lemma 2.3]{hyry}, we get that 
\begin{eqnarray}\label{vanishing of lc 12,3+}
 [H^i_{T_{1,2}^++T_3^+}(M)]_{\bf n}=0 \mbox{ for all }n_3 \geq 0 \mbox{ and all }i \geq 0.
\end{eqnarray}
Consider 
the Mayer-Vietoris sequence of local cohomology modules 
\begin{eqnarray*}
 \cdots \longrightarrow H^i_{T_1^++T_2^+}(M) \longrightarrow H^i_{T_1^+}(M) \oplus H^i_{T_2^+}(M) 
 \longrightarrow H^i_{T_{1,2}^+}(M) \longrightarrow H^{i+1}_{T_1^+ +T_2^+}(M) \longrightarrow \cdots.
\end{eqnarray*}
Then, by equations (\ref{vanishing of lc 1+}), (\ref{vanishing of lc 12+}) and 
(\ref{vanishing of lc i+}), for all $i \geq 0$, 
\begin{eqnarray}\label{vanishing of lc 1,2+}
[H^i_{T_{1,2}^+}(M)]_{\bf n}=0 \mbox{ if } n_1\geq 0 \mbox{ and }n_2 \geq 0.
\end{eqnarray}
Again considering the Mayer-Vietoris sequence of local cohomology modules
\begin{eqnarray*}
 \cdots \longrightarrow H^i_{T_{1,2}^++T_3^+}(M) \longrightarrow H^i_{T_{1,2}^+}(M) \oplus H^i_{T_3^+}(M) 
 \longrightarrow H^i_{T_{++}}(M) \longrightarrow H^{i+1}_{T_{1,2}^+ +T_3^+}(M) \longrightarrow \cdots
\end{eqnarray*}
and using equations (\ref{vanishing of lc i+}), (\ref{vanishing of lc 1,2+}) and (\ref{vanishing of lc 12,3+}), we get that 
$[H^i_{T_{++}}(M)]_{\bf n}=0$ for all ${\bf n} \in \mathbb N^3 $ and $i \geq 0$.
\end{proof}

Suppose $T$ is an $\mathbb N$-graded ring defined over a local ring 
and $\mathcal{M}$ is the maximal homogeneous ideal of $T$. For $\mathbb{N}-$graded $T-$module $M$
let $$a(M):=\sup\{k\mid [H^{\dim M}_\mathcal M(M)]_k \neq0\}.$$ M. Herrmann, E. Hyry and J. Ribbe 
proved that for homogeneous ideal $I$ of positive height in a multi-graded ring $B$, 
$a(\mathcal{R}(I))=-1$ \cite{hhr}. The same proof shows the following:

\begin{lemma} \label{lemma3}\rm\cite[Lemma 2.1]{hhr}
 Let $B$ be a multi-graded ring of dimension $d$ defined over a local ring and 
let $I \subseteq B$ be a homogeneous ideal of positive height. Let $\F=\{I_n\}$ be an $I-$admissible 
filtration of homogeneous ideal in $B$ such that $\mathcal R(\F)$ is a 
finite $\mathcal R(I)-$module. Then $a(\mathcal R(\F))=-1$.
\end{lemma}
\begin{proof}
Let $\mathfrak n$ be the maximal homogeneous ideal of $B$. Let $\mathcal M=\mathfrak n \bigoplus_{n >0}I^nt^n$ 
be a maximal homogeneous ideal in $\mathcal R(I)$. Then 
$[H^i_{\mathcal M}(\mathcal R(\F))]_n=0$ if and only if 
$[H^i_{\mathcal M \otimes_B B_\mathfrak n}(\mathcal R(\F) \otimes_B B_{\mathfrak n})]_n=0$. 
Since $\mathcal R(\F) \otimes_B B_\mathfrak n = \mathcal R_{B_\mathfrak n}(\F B_\mathfrak n)$, 
we may assume that $B$ is local with maximal ideal $\mathfrak n$. \\
Let $\mathcal R(\F)_+=\bigoplus_{n>0}I_nt^n$ and 
$ G(\F)=\bigoplus_{n\geq 0}I_n/I_{n+1}$. For an $\mathbb N$-graded module 
$M$ let $M(n)$ denote the graded module $M$ with $M(n)_k=M_{n+k}$. Then we have the exact sequences 
$$ 0 \longrightarrow \mathcal{R}(\F)_+\longrightarrow \mathcal{R}(\F) \longrightarrow B \rightarrow 0$$
$$0 \longrightarrow \mathcal{R}(\F)_+(1) \longrightarrow \mathcal{R}(\F) \longrightarrow G(\F) \longrightarrow 0.$$
Therefore we get the long exact sequence of local cohomology modules
$$\cdots \longrightarrow H^i_\mathcal M(\mathcal R(\F)_+) \longrightarrow H^i_\mathcal M(\mathcal R(\F)) 
\longrightarrow H^i_\mathcal{M}(B) \longrightarrow \cdots \mbox{ and }$$
$$\cdots \longrightarrow H^i_\mathcal M(\mathcal R(\F)_+(1)) \longrightarrow H^i_\mathcal M(\mathcal R(\F)) 
\longrightarrow H^i_\mathcal{M}(G(\F))\longrightarrow \cdots .$$
Note that $H^i_{\mathcal M}(B)=H^i_\mathfrak n(B)$. Hence $[H^i_\mathcal M(B)]_n=0$ for $n \neq 0$. 
This gives the isomorphisms 
\begin{equation} \label{equation5.2}
[H^{d+1}_\mathcal M(\mathcal R(\F)_+)]_n \longrightarrow [H^{d+1}_\mathcal M(\mathcal R(\F))]_n \mbox{ for }n \neq 0 
\end{equation}
and the epimorphisms
\begin{equation} \label{equation5.3}
[H^{d+1}_\mathcal M(\mathcal R(\F)_+)]_{n+1} \longrightarrow [H^{d+1}_\mathcal M(\mathcal R(\F))]_n \mbox{ for }n \in \mathbb Z.
\end{equation}
Since $H^{d+1}_\mathcal M(\mathcal R(\F)_+)$ is Artinian, $[H^{d+1}_\mathcal M(\mathcal R(\F)_+)]_n=0$ 
for $n\gg0$. Therefore, by equations (\ref{equation5.2}) and (\ref{equation5.3}), $[H^{d+1}_\mathcal{M}(\mathcal R(\F))]_n=0$ 
for $n\geq 0$. If $[H^{d+1}_\mathcal M(\mathcal R(\F))]_{-1}=0$ then 
$H^{d+1}_\mathcal M(\mathcal R(\F))=0,$ a contradiction. Thus 
$a(\mathcal R(\F))=-1$.
\end{proof}

For a standard $\mathbb{N}^g$-graded ring $T$ defined over a local ring and finitely generated 
$\mathbb{N}^g$-graded $T$-module $M$, let 
\begin{eqnarray*}
 a^i(M)&=&\sup\{k\in \mathbb{Z}\mid [H^{\dim M}_{\mathcal{M}}(M)]_{{\bf n}} \neq 0 \mbox{ for some } 
 {\bf n}\in \mathbb{Z}^{g} \mbox{ with }n_i=k\}
\end{eqnarray*}
be $a-$invariants of $M$ \cite{hyry}. Here $\mathcal{M}$ is the maximal homogeneous 
ideal of $T$.
E. Hyry proved that $a^i(\mathcal R({\bf I}))=-1$ for all $1\leq i \leq g$ 
\cite[Lemma 2.7]{hyry}. In following lemma we prove that $a^i(\mathcal R(\mathcal F))=-1$ for all $1\leq i \leq g$. 

\begin{lemma} \label{negativity of a-invariants}
 Let $(R,\m)$ be a Noetherian local ring and $I_1,\ldots,I_g$ be $\m-$primary ideals in $R$. Let 
 $\F$ be an ${\bf I}=(I_1,\ldots,I_g)-$admissible filtration. Then 
 $a^i(\mathcal R(\mathcal F))=-1$ for all $1\leq i \leq g$.
\end{lemma}
\begin{proof}
Let $\mathcal G=\{\F({\bf n}-n_i{\bf e}_i)\}$ be $\mathbb N^{g-1}-$graded 
$(I_1,\ldots,\widehat{I_i},\ldots,I_g)-$filtration. 
Then $B=\mathcal R(\mathcal G)$ is $\mathbb N^{g-1}-$graded ring and 
$I=\bigoplus_{{\bf m}\in \mathbb N^g}
{\bf I}^{{\bf m}-(m_i-1){\bf e}_i}$
is homogeneous ideal in $B$. Also, $\mathcal H=\{\mathcal H(n)=\bigoplus_{{\bf m} \in 
\mathbb N^{g},m_i=n} \F({\bf m})\}_{n \in \mathbb N}$ is an $I-$filtration in $B$ and 
$\R_B(\mathcal H)=\R(\F)$ is a finite
$\R_B(I)=\R(I)-$module. Therefore by Lemma \ref{lemma3}, 
$a(\mathcal R_B(\mathcal H))=-1$. Consider $\phi: \mathbb Z^g \longrightarrow \mathbb Z $ 
defined as $\phi({\bf n})=n_i$. Then 
$[H^{d+1}_\mathcal M(\R_B(\mathcal H))]_n=[H^{d+1}_{\mathcal M^\phi}(\R(\F))^\phi]_n=\bigoplus_
{{\bf m},m_i=n}[H^i_\mathcal M(M)]_{\bf m}$. Hence $a^i(\R(\F))=-1$.
\end{proof}

\begin{lemma} \label{local cohomology of R and R*}
Let $S^*$ be a $\mathbb Z^g-$graded ring and let 
$S=\bigoplus_{{\bf n}\in \mathbb N^g} \hookrightarrow S^*$ 
be an inclusion. Then 
\end{lemma}
\begin{enumerate}
 \item for $i >1$, $H^i_{S_{++}}(S) \cong H^i_{S_{++}}(S^*)$.
\item We have an exact sequence
$$0 \longrightarrow H^0_{S_{++}}(S) \longrightarrow H^0_{S_{++}}(S^*) \longrightarrow S^*/S 
\longrightarrow H^1_{S_{++}}(S) \longrightarrow H^1_{S_{++}}(S^*) \longrightarrow 0.$$
 \end{enumerate}
\begin{proof}
 We have an exact sequence 
 $$0 \longrightarrow S \longrightarrow S^* \longrightarrow S^*/S \longrightarrow 0.$$
 Since $S^*/S$ is $S_{++}-$torsion, $H^0_{S_{++}}(S^*/S)=S^*/S$ and 
 $H^i_{S_{++}}(S^*/S)=0$ for all $i>0$. Considering the long exact sequence 
 of local cohomology modules result follows. 
\end{proof}

\begin{theorem}
 \label{good joint reductions exist in CMcase}
Let $(R,\m)$ be a Cohen-Macaulay local ring of dimension $3$ with infinite residue field 
 and $I,J,K$ be $\m-$primary ideals in $R$. 
 Let $\mathcal F=\{\mathcal F{(r,s,t)}\}$ be an $(I,J,K)-$admissible filtration. 
 Assume that $\F({\bf n})=\F({\bf n}-n_i{\bf e}_i)$ for $n_i \leq 0$. 
 Suppose 
 $\mathcal R(\mathcal F)$ 
 is Cohen-Macaulay. Then there exists a good complete reduction $(M_{ij})$, where
$M= \begin{bmatrix}
  x_1 & x_2 & x_3\\
  y_1 & y_2 & y_3\\
  z_1 & z_2 & z_3\\
 \end{bmatrix}
$,
such that $(M_{i,\sigma (i)})$ is a good joint reduction of $\mathcal F$ for each permutation $\sigma \in S_3$.
\end{theorem}
\begin{proof}
By Lemma \ref{negativity of a-invariants}, $a^j(\mathcal R(\mathcal F))=-1$ for $j=1,2,3$. 
Since $\R(\F)$ is Cohen-Macaulay, for each $j=1,2,3$ we have 
$$[H^i_{\mathcal M}(\R(\F))]_{\bf n}=0 \mbox{ for all }i \geq 0 \mbox{ and }n_j \geq 0,$$
where $\mathcal M$ is the maximal homogeneous ideal of $\R(I,J,K)$. Therefore, 
by Lemma \ref{vanishing of lc T++}, for all $i \geq 0$
$$[H^i_{\mathcal R_{++}}(\R(\F))]_{\bf n}=0 \mbox{ for all }{\bf n} \in \mathbb N^3.$$
Hence, by Lemma \ref{local cohomology of R and R*}, for all $i \geq 0$
$$[H^i_{\mathcal R_{++}}(\R^\prime(\F))]_{\bf n}=0 \mbox{ for all }{\bf n} \in \mathbb N^3.$$
Consider the exact sequence
\begin{eqnarray*}
 0 \longrightarrow \R^\prime(\F)({\bf e}_i) \buildrel t_i^{-1}\over \longrightarrow 
 \R^\prime(\F) \longrightarrow G_i^\prime(\F) \longrightarrow 0,
\end{eqnarray*}
where $G_i^\prime(\F)=\R^\prime(\F)/t_i^{-1}\R^\prime(\F)$.
This gives a long exact sequence of local cohomology modules 
$$\cdots \longrightarrow [H^j_{\R_{++}}(\R^\prime(\F))]_{{\bf n}+{\bf e}_i} 
\longrightarrow [H^j_{\mathcal R_{++}}
(\R^\prime(\F))]_{\bf n} \longrightarrow [H^j_{\mathcal R_{++}}(G_i^\prime(\mathcal F))]_{\bf n} 
\longrightarrow [H^{j+1}_{\R_{++}}(\R^\prime(\F))]_{{\bf n}+{\bf e}_i}\longrightarrow \cdots.$$
Hence $[H^i_{\R_{++}}(G_i^\prime(\F))]_{\bf n}=0$ for all $i \geq 0$ and ${\bf n} +{\bf e}_i \geq {\bf 0}$. 
Since $\F({\bf n})=\F({\bf n}-n_i{\bf e}_i)$ for $n_i \leq 0$, 
$[G_i^\prime(\F)]_{\bf n}=0$ for $n_i <0$. Hence $\left[\frac{G_i^\prime(\F)}{G_i(\F)}\right]_{\bf n}=0$ 
if ${\bf n} \geq 0$ or $n_i < 0$. Therefore, by Lemma \ref{local cohomology of R and R*}, 
$[H^i_{\R_{++}}(G_i(\F))]_{\bf n}=[H^i_{\R_{++}}(G_i^\prime(\F))]_{\bf n}$ 
for all $i \geq 0$ if ${\bf n}\geq {\bf 0}$ or $n_i<0$. Hence 
$[H^i_{\R_{++}}(G_i(\F))]_{\bf n}=0$ for all $i \geq 0$ and ${\bf n}+{\bf e}_i \geq {\bf 0}$.
Hence result follows from Theorem \ref{existence of good joint reductions}.
\end{proof}

\begin{example}
 Let $R=k[x,y,z], I=(x,y,z)$ and $J=K=(x^2,y,z)$. Then $(x,y,z)$ forms a good 
 joint reduction of $\{I^rJ^sK^t\}$. 
\end{example}

\section{Computation of $\lm_R\left(\frac{\ov{I^rJ^sK^t}}{a^r\ov{J^sK^t}+b^s\ov{I^rK^t}+c^t\ov{I^rJ^s}}\right)$}

Let $(R,\m)$ be an analytically unramified Cohen-Macaulay local ring 
of dimension $3$ and $I,J,K$ be $\m-$primary ideals in $R$. Let $(a,b,c)$ be a 
good joint reduction of $\{\ov{I^rJ^sK^t}\}$.
In this section we express the length 
$$\lm_R\left(\frac{\ov{I^rJ^sK^t}}{a^r\ov{J^sK^t}+b^s\ov{I^rK^t}+c^t\ov{I^rJ^s}}\right)$$
in terms of the normal Hilbert coefficients. The main tool employed for this calculation 
is the homology of a multigrades version of the  Kirby-Mehran complex \cite{km}.

\begin{lemma} 
\label{properties of good joint reduction}
Let $(R,\m)$ be a Noetherian local ring and $I,J$ be ideals in $R$. Let $a\in I$, 
$b \in J$ and $(a,b)$ be a regular sequence. 
 \begin{enumerate}
   \item \label{properties of good joint reduction1} Suppose $a $ satisfies 
   \begin{eqnarray}
    \label{good joint reduction condition1}
 (a) \cap \ov{I^rJ^sK^t}&=&a\ov{I^{r-1}J^sK^t} \mbox{ for all }r>0 \mbox{ and all }s,t \geq 0.
   \end{eqnarray}
Then for all $r\geq 1$ and $s,t \geq 0$, 
  \begin{eqnarray*}
  (a^m) \cap \ov{I^rJ^sK^t}=a^m \ov{I^{r-m}J^sK^t}\mbox{ for }0 < m \leq r.
  \end{eqnarray*}
\item \label{properties of good joint reduction2} Suppose $(a,b)$ satisfies 
\begin{eqnarray} \label{good joint reduction conditions}
  \label{good joint reduction condition2}
 (a,b) \cap \ov{I^rJ^sK^t}&=&a\ov{I^{r-1}J^sK^t}+b\ov{I^rJ^{s-1}K^t} \mbox{ for all }r,s >0 
 \mbox{ and all }t \geq 0.
\end{eqnarray}
Then for all $r,s \geq 1$ and $t \geq 0$, 
$$(a^m,b^n) \cap \ov{I^rJ^sK^t}=a^m\ov{I^{r-m}J^sK^t}+b^n\ov{I^rJ^{s-n}K^t} \mbox{ for }0 <m \leq r, 
0 < n \leq s.$$
 \end{enumerate}
 \end{lemma}
\begin{proof}
\begin{enumerate}
 \item Induct on $m$. From equation (\ref{good joint reduction condition1}), result follows for $m=1$. 
 Let $m+1 \leq r$ and $a^{m+1}u \in \ov{I^rJ^sK^t}$. Then $a^{m+1}u \in 
 (a^m) \cap \ov{I^rJ^sK^t}=
 a^m \ov{I^{r-m}J^sK^t}$, by induction hypothesis. Hence $ au \in \ov{I^{r-m}J^sK^t} \cap (a)=
 a \ov{I^{r-m-1}J^sK^t}$. Thus $u \in \ov{I^{r-m-1}J^sK^t}$. 

\item First we use induction on $m$ to prove that 
$(a^m,b) \cap \ov{I^rJ^sK^t}=a^{m}\ov{I^{r-m}J^sK^t}+b\ov{I^rJ^{s-1}K^t}$ for $0 <m \leq r$. 
From equation (\ref{good joint reduction condition2}), the result follows for $m=1$. 
Let $1<m \leq r$ and 
$a^mu+bv \in \ov{I^rJ^sK^t}$. Thus $a^mu+bv \in \ov{I^rJ^sK^t} \cap (a^{m-1},b)=
a^{m-1}\ov{I^{r-(m-1)}J^sK^t}+b\ov{I^rJ^{s-1}K^t}$, by induction hypothesis. 
Let $a^mu+bv=a^{m-1}p+bq $ for some $p\in \ov{I^{r-(m-1)}J^sK^t}$ and 
$q \in \ov{I^rJ^{s-1}K^t}$. Then $a^{m-1}(au-p)\in (b)$. Hence 
$au -p \in (b)$ which implies that $p\in(a,b) \cap \ov{I^{r-(m-1)}J^sK^t}=
a\ov{I^{r-m}J^sK^t}+b\ov{I^{r-(m-1)}J^{s-1}K^t}$. Let $p=ap_1+bp_2$ 
for some $p_1 \in \ov{I^{r-m}J^sK^t}$ and $p_2 \in \ov{I^{r-(m-1)}J^{s-1}K^t}$. 
Thus $a^mu+bv=a^mp_1+b(a^{m-1}p_2+q) \in a^m\ov{I^{r-m}J^sK^t}+b \ov{I^rJ^{s-1}K^t}$. \\
Similar argument shows that $(a,b^n)\cap \ov{I^rJ^sK^t}=a\ov{I^{r-1}J^sK^t}+b^n
\ov{I^rJ^{s-n}K^t}$ for $0 <n \leq s$. \\
To prove the assertion we use induction on $m+n$. From equation (\ref{good joint reduction condition2}), 
the result follows for $m=n=1$. Let $m+n > 2$. We may assume that $n> 1$. Let 
$a^mu+b^nv \in (a^m,b^n) \cap \ov{I^rJ^sK^t}$. Then $a^mu+b^nv=a^mp+bq$ for 
some $p \in \ov{I^{r-m}J^sK^t}$ and $q \in \ov{I^rJ^{s-1}K^t}$. 
Hence $b(b^{n-1}v-q) \in (a^m)$ which implies that $b^{n-1}v-q \in (a^m)$. Thus 
$q \in (a^m,b^{n-1}) \cap \ov{I^{r}J^{s-1}K^t}=a^m\ov{I^{r-m}J^{s-1}K^t}+
b^{n-1}\ov{I^{r}J^{s-n}K^t}$, by induction. Let $q=a^{m}q_1+b^{n-1}q_2$ for some 
$q_1 \in \ov{I^{r-m}J^{s-1}K^t}$ and $q_2 \in \ov{I^{r}J^{s-n}K^t}$. 
Thus $a^mu+b^nv=a^m(p+bq_1)+b^n(q_2)\in a^m\ov{I^{r-m}J^sK^t}+b^n\ov{I^rJ^{s-n}K^t}$. 
\end{enumerate}
\end{proof}

We now introduce an $\NN^3$-graded version of the Kirby-Mehran complex associated to the joint reduction $(a,b,c)$ of the filtration $\{\ov{I^rJ^sK^t}\}:$
 $$
C_{\cdot}((a,b,c),r,s,t): 0 \longrightarrow \frac{R}{\ov{I^r}} \oplus \frac{R}{\ov{J^s}} \oplus 
\frac{R}{\ov{K^t}}\buildrel \phi_{1} \over\longrightarrow \frac{R}{\ov{I^rJ^s}} \oplus 
\frac{R}{\ov{I^rK^t}} \oplus \frac{R}{\ov{J^sK^t}} 
\buildrel \phi_0 \over\longrightarrow \frac{R}{\ov{I^rJ^sK^t}} \longrightarrow
0,$$
where $\phi_0$ and $\phi_1$ are defined as
$$\phi_1(u^\prime,v^\prime,w^\prime)=((a^rv+b^su)^\prime,(a^rw-c^tu)^\prime,(-b^sw-c^tv)^\prime),
\phi_0(u^\prime,v^\prime,w^\prime)=(c^tu+b^sv+a^rw)^\prime.$$
Here $^\prime$ denotes image of an element in respective quotients. 
Let $H_i((a,b,c),r,s,t)$ denote the $i$th homology of the complex $C_{\cdot}((a,b,c),r,s,t).$

\begin{proposition}
\label{computation of homology}
Let $(R,\m)$ be an analytically unramified Cohen-Macaulay local ring 
of dimension $3$ with infinite residue field and $I,J,K$ be $\m-$primary ideals in $R$. Let $(a,b,c)$ be a 
good joint reduction of $\{\ov{I^rJ^sK^t}\}$. Then for integers $r, s,t>0 $
\begin{enumerate}
\item $H_0((a,b,c),r,s,t) \cong \frac{R}{(a^r,b^s,c^t)+\ov{I^{r}J^{s}K^t}}$
\item $H_1((a,b,c),r,s,t) \cong \frac{\ov{I^rJ^sK^t} \cap (a^r,b^s,c^t)}{a^r\ov{J^sK^t}+b^s\ov{I^rK^t}+c^t\ov{I^rJ^s}}$
\item $H_2((a,b,c),r,s,t)=0.$
\end{enumerate} 
\end{proposition}
\begin{proof}
\begin{enumerate}
 \item Since $\im \phi_0=\frac{(a^r,b^s,c^t)+\ov{I^rJ^sK^t}}{\ov{I^rJ^sK^t}},$
$$H_0((a,b,c),r,s,t) \cong \frac{R}{(a^r,b^s,c^t)+\ov{I^{r}J^{s}K^t}}.$$
\item Consider the commutative diagram:
$$
\CD
0@>   >> \im \phi_1 @>  >> \frac{R}{\ov{I^rJ^s}} \oplus \frac{R}{\ov{I^rK^t}} \oplus \frac{R}{\ov{J^sK^t}} @> \delta >>\frac{(a^r,b^s,c^t)}{a^r\ov{J^{s}K^t}+b^s\ov{I^rK^t}+c^t\ov{I^rJ^s}} @>  >> 0 \\
@.     @VVV             @V\psi VV                @V \gamma VV   @.  \\
0@>   >> \ker \phi_0 @>  >> \frac{R}{\ov{I^rJ^s}} \oplus \frac{R}{\ov{I^rK^t}} \oplus \frac{R}{\ov{J^sK^t}} @> \phi_0
 >> \frac{(a^r,b^s,c^t)+\ov{I^rJ^sK^t}}{\ov{I^rJ^sK^t}} @> >> 0
\endCD
$$
where $\delta(u^\prime,v^\prime,w^\prime)=(c^tu+b^sv+a^rw)^\prime$ and $\gamma$ is the natural map. 
To prove exactness of top row it is 
enough to prove exactness at $\frac{R}{\ov{I^rJ^s}} \oplus \frac{R}{\ov{I^rK^t}} \oplus \frac{R}{\ov{J^sK^t}}$. 
Suppose $\delta(u^\prime,v^\prime,w^\prime)=0$. Then $c^tu+b^sv+a^rw=c^tx+b^sy+a^rz$ 
for some $x \in \ov{I^rJ^s}, y\in \ov{I^rK^t}$ and $z \in \ov{J^sK^t}$. Hence 
$u-x =b^sx_1+a^rx_2$ for some $x_1,x_2 \in R$. Similarly, $v-y=a^ry_1+c^ty_2$ and 
$w-z=b^sz_1+c^tz_2$ for some $y_1,y_2,z_1,z_2 \in R$. 
Thus 
$$c^t(b^sx_1+a^rx_2)+b^s(a^ry_1+c^ty_2)+a^r(b^sz_1+c^tz_2)=0. $$
Hence $x_1+y_2 \in (a^r),x_2+z_2 \in (b^s),y_1+z_1 \in (c^t).$ 
Therefore $y_2^\prime=-x_1^\prime $ in $\frac{R}{\ov{I^r}}$ $,z_1^\prime=-y_1^\prime$ in 
$\frac{R}{\ov{K^t}}$ and $z_2^\prime=-x_2^\prime$ in $\frac{R}{\ov{J^s}}$. 
Hence $\phi_1(x_1^\prime,x_2^\prime,y_1^\prime)=(u^\prime,v^\prime,w^\prime).$ 
Thus $(u^\prime,v^\prime,w^\prime) \in \im \phi_1$.\\
Using the snake lemma and the fact that $\psi$ is an isomorphism we obtain,
$$ H_1((a,b,c),r,s,t)=\frac{\ker\phi_0}{\im \phi_1} \cong  \ker \gamma  
 = \frac{\ov{I^rJ^sK^t} \cap (a^r,b^s,c^t)}{a^r\ov{J^sK^t}+b^s\ov{I^rK^t}+c^t\ov{I^rJ^s}}.$$
\item Suppose $\phi_1(u^\prime,v^\prime,w^\prime)=0$. 
Thus $((a^rv+b^su)^\prime,(a^rw-c^tu)^\prime,(-b^sw-c^tv)^\prime)=0$. 
Hence $a^rv+b^su \in \ov{I^rJ^s} \cap(a^r,b^s)=a^r\ov{J^s}+b^s\ov{I^r}$, by Lemma 
\ref{properties of good joint reduction}. 
Let $a^rv+b^su=a^rp+b^sq$ for some $p\in \ov{J^s}$ and $q \in \ov{I^r}$. 
Thus $v-p \in (b^s)$ and $u-q \in (a^r)$. Hence $u \in \ov{I^r}$ and 
$v \in \ov{J^s}$. Therefore $u^\prime=v^\prime=0$. Similarly, $w^\prime=0$.
\end{enumerate}
\end{proof}

\begin{proposition}
\label{computation of some length}
Let $(R,\m)$ be an analytically unramified Cohen-Macaulay local ring 
of dimension $3$ with infinite residue field and $I,J,K$ be $\m-$primary ideals in $R$. Let $(a,b,c)$ be a 
good joint reduction of $\{\ov{I^rJ^sK^t}\}$. Then for $r,s,t>0$,
\begin{eqnarray*}
&&\lm_R\left(\frac{(a^r,b^s,c^t)}{a^r\ov{J^sK^t}+b^s\ov{I^rK^t}+c^t\ov{I^rJ^s}}\right)\\
&=&
\left[\lm_R\left(\frac{R}{\ov{I^rJ^s}}\right)+\lm_R\left(\frac{R}{\ov{I^rK^t}}\right)+\lm_R\left(\frac{R}{\ov{J^sK^t}}\right)\right]-
\left[\lm_R\left(\frac{R}{\ov{I^r}}\right)+\lm_R\left(\frac{R}{\ov{J^s}}\right)+\lm_R\left(\frac{R}{\ov{K^t}}\right)\right]
\end{eqnarray*}
\end{proposition}
\begin{proof}
 From the complex $C_{\cdot}((a,b,c),r,s,t),$ for $r,s,t>0,$ we have
\begin{eqnarray*}
&& \lm_R\left(\frac{R}{\ov{I^rJ^sK^t}}\right)-\left[\lm_R\left(\frac{R}{\ov{I^rJ^s}}\right)+
\lm_R\left(\frac{R}{\ov{I^rK^t}}\right)+\lm_R\left(\frac{R}{\ov{J^sK^t}}\right)\right]+
\left[\lm_R\left(\frac{R}{\ov{I^r}}\right)+\lm_R\left(\frac{R}{\ov{J^s}}\right)+\lm_R\left(\frac{R}{\ov{K^t}}\right)\right]\\
&=&\lm_R(H_0((a,b,c),r,s,t))-\lm_R(H_1((a,b,c),r,s,t))+\lm_R(H_2((a,b,c),r,s,t)).
\end{eqnarray*}
Therefore, by Proposition \ref{computation of homology}, 
\begin{eqnarray*}
 &&\left[\lm_R\left(\frac{R}{\ov{I^rJ^s}}\right)+\lm_R\left(\frac{R}{\ov{I^rK^t}}\right)+\lm_R\left(\frac{R}{\ov{J^sK^t}}\right)\right]-
 \left[\lm_R\left(\frac{R}{\ov{I^r}}\right)+\lm_R\left(\frac{R}{\ov{J^s}}\right)+\lm_R\left(\frac{R}{\ov{K^t}}\right)\right]\\
&=&\lm_R\left(\frac{R}{\ov{I^rJ^sK^t}}\right)-\lm_R\left(\frac{R}{(a^r,b^s,c^t)+\ov{I^{r}J^{s}K^t}}\right)+
\lm_R\left(\frac{\ov{I^rJ^sK^t} \cap (a^r,b^s,c^t)}{a^r\ov{J^sK^t}+b^s\ov{I^rK^t}+c^t\ov{I^rJ^s}}\right)\\
&=&\lm_R\left(\frac{(a^r,b^s,c^t)+\ov{I^rJ^sK^t}}{\ov{I^rJ^sK^t}}\right)+\lm_R\left(\frac{\ov{I^rJ^sK^t} \cap (a^r,b^s,c^t)}{a^r\ov{J^sK^t}+b^s\ov{I^rK^t}+c^t\ov{I^rJ^s}}\right)\\
&=&\lm_R\left(\frac{(a^r,b^s,c^t)}{\ov{I^rJ^sK^t} \cap (a^r,b^s,c^t)}\right)+\lm_R\left(\frac{\ov{I^rJ^sK^t} \cap (a^r,b^s,c^t)}{a^r\ov{J^sK^t}+b^s\ov{I^rK^t}+c^t\ov{I^rJ^s}}\right)\\
&=&\lm_R\left(\frac{(a^r,b^s,c^t)}{a^r\ov{J^sK^t}+b^s\ov{I^rK^t}+c^t\ov{I^rJ^s}}\right).
\end{eqnarray*}
\end{proof}

In an analytically unramified local ring of dimension $d$ there exists a 
polynomial $\ov{P}_{I,J}(x,y)\in \mathbb Q[x,y]$ 
(respectively $\ov{P}_{I,J,K}(x,y,z) \in \mathbb Q[x,y,z]$) of total degree $d$, 
called as the {\it normal Hilbert 
polynomial of $I,J$} (respectively {\it normal Hilbert polynomial of $I,J,K$}), 
such that $\ov{P}_{I,J}(r,s)=\lm(R/\ov{I^rJ^s})$ (respectively $\ov{P}_{I,J,K}(r,s,t)=
\lm(R/\ov{I^rJ^sK^t})$) for $r,s \gg 0$ (respectively $r,s,t \gg 0$). We write 
\begin{eqnarray*}
 \ov{P}_{I,J}(x,y)&=&\sum_{i+j \leq
d}(-1)^{d-(i+j)}\ov{e}_{(i,j)}(I,J)\binom{x+i-1}{i}\binom{y+j-1}{j} \mbox{ and }\\
 \ov{P}_{I,J,K}(x,y,z)&=&\sum_{i+j+k \leq
d}(-1)^{d-(i+j+k)}\ov{e}_{(i,j,k)}\binom{x+i-1}{i}\binom{y+j-1}{j}\binom{z+k-1}{k}.
\end{eqnarray*} 
Rees proved that $\ov{e}_{(1,0)}(I,J)=\ov{e}_1(I)$ in an analytically unramified Cohen-Macaulay 
local ring of dimension $2$ \cite[Theorem 1.2]{rees}. In following lemma we derive a similar relation between 
coefficients of degree $2$ in $\ov{P}_{I,J,K}(x,y,z)$, $ \ov{P}_{I,J}(x,y)$ and $\ov{P}_I(x)$ 
in an analytically unramified Cohen-Macaulay local ring of dimension $3.$

\begin{lemma} \label{lemma1}
Let $(R,\m)$ be an analytically unramified Cohen-Macaulay local ring 
of dimension $3$ and $I,J,K$ be $\m-$primary ideals in $R$. Then we have
\begin{center}
\begin{tabular}{|c|c|c|}
\hline$\ov{e}_{(2,0,0)}=\ov{e}_1(I)$  & $\ov{e}_{(0,2,0)}=\ov{e}_1(J)$ & $\ov{e}_{(0,0,2)}=\ov{e}_1(K)$\\
\hline
$\ov{e}_{(1,1,0)}=\ov{e}_{(1,1)}(I,J)$  & $\ov{e}_{(0,1,1)}=\ov{e}_{(1,1)}(J,K)$ & $\ov{e}_{(1,0,1)}=\ov{e}_{(1,1)}(I,K).$\\
\hline
\end{tabular}
\end{center}
\end{lemma}
\begin{proof}
Let $t>0$ be fixed. There exists $c \in K$ such that 
$$(c) \cap \ov{I^rJ^sK^t}=c\ov{I^rJ^sK^{t-1}} \mbox{ for all }r,s \geq 0 \mbox{ and }t>0.$$
See \cite{rees}. Then, by Lemma \ref{properties of good joint reduction}
(\ref{properties of good joint reduction1}), 
$(c^t) \cap \ov{I^rJ^sK^t}=c^t\ov{I^rJ^s}$ for all $t >0$ and $r,s \geq 0$. Therefore we have the exact sequence
$$
0 \longrightarrow \frac{\ov{I^rJ^s}}{\ov{I^rJ^sK^t}} \longrightarrow \frac{R}{\ov{I^rJ^sK^t}}
\buildrel \mu_{c^t} \over\longrightarrow \frac{R}{\ov{I^rJ^sK^t}} \longrightarrow \frac{R}{\ov{I^rJ^sK^t}+(c^t)} \longrightarrow
0,$$ 
Therefore 
$\lm_R\left(\frac{\ov{I^rJ^s}}{\ov{I^rJ^sK^t}}\right)=\lm_R\left(\frac{R}{\ov{I^rJ^sK^t}+(c^t)}\right)$. 
Let $R^\prime=R/(c^t)$. Consider the filtration $\F=\{\F(r,s):r,s \in \mathbb Z\}$, where 
$\mathcal F{(r,s)}=\left\{\frac{\ov{I^rJ^sK^t}+(c^t)}{(c^t)}\right\}$. 
Let $``\prime$'' denotes image of an ideal in $R^\prime$. 
By \cite[Theorem 2.4]{rees3}, $\lm_R\left(\frac{R}{\ov{I^rJ^sK^t}+(c^t)}\right)$ is a 
polynomial of total degree $2$ for $r,s \gg0$ with
$$\lm_R\left(\frac{R}{\ov{I^rJ^sK^t}+(c^t)}\right)=e(I^\prime)\binom{r+1}{2}+\ov{e}_{(1,1)}(I^\prime,J^\prime)rs+e(J^\prime)\binom{s+1}{2}+\mbox{ lower degree terms}.
$$  
Note that $e(I^\prime)=t\ov{e}_{(2,0,1)},e(J^\prime)=t\ov{e}_{(0,2,1)}$ and $\ov{e}_{(1,1)}(I^\prime,J^\prime)
=t\ov{e}_{(1,1,1)}$. Therefore  
for $r,s \gg 0$ and $t>0$, 
$$\lm_R\left(\frac{R}{\ov{I^rJ^sK^t}+(c^t)}\right)=t\ov{e}_{(2,0,1)}\binom{r+1}{2}+t\ov{e}_{(1,1,1)}rs+
t\ov{e}_{(0,2,1)}\binom{s+1}{2}+\mbox{ lower degree terms}.
$$  
We have
$$\lm_R\left(\frac{R}{\ov{I^rJ^sK^t}}\right)=\lm_R\left(\frac{R}{\ov{I^rJ^s}}\right)
+\lm_R\left(\frac{\ov{I^rJ^s}}{\ov{I^rJ^sK^t}}\right).$$
Fixing $t \gg0$ and comparing the coefficients of degree $2$ in $r,s$, we get 
$$\ov{e}_{(2,0,0)}=\ov{e}_{(2,0)}(I,J),\ov{e}_{(1,1,0)}=\ov{e}_{(1,1)}(I,J),\ov{e}_{(0,2,0)}=\ov{e}_{(0,2)}(I,J).$$ 
By \cite[Theorem 4.3.1]{masuti}, 
$\ov{e}_{(2,0)}(I,J)=\ov{e}_1(I)$ and $\ov{e}_{(0,2)}(I,J)=\ov{e}_1(J)$. 
See also \cite[Theorem 1.2]{rees}. Hence 
$\ov{e}_{(2,0,0)}=\ov{e}_1(I),\ov{e}_{(0,2,0)}=\ov{e}_1(J)  $ and 
$\ov{e}_{(1,1,0)}=\ov{e}_{(1,1)}(I,J)$. \\
Similar argument shows the other equalities.
\end{proof}

\begin{theorem}\label{theorem1}
 Let $(R,\m)$ be an analytically unramified Cohen-Macaulay local ring 
of dimension $3$ with infinite residue field and $I,J,K$ be $\m-$primary ideals in $R$. Let $(a,b,c)$ be a 
good joint reduction of $\{\ov{I^rJ^sK^t}\}$. Then for $r,s, t \gg 0$,
\begin{eqnarray*}
 \lm_R\left(\frac{\ov{I^rJ^sK^t}}{a^r\ov{J^sK^t}+b^s\ov{I^rK^t}+c^t\ov{I^rJ^s}}\right)
&=&r[\ov{e}_{(1,0)}(I,J)+\ov{e}_{(1,0)}(I,K)-\ov{e}_{(1,0,0)}-\ov{e}_2(I)]\\&+&
s[\ov{e}_{(0,1)}(I,J)+\ov{e}_{(1,0)}(J,K)-\ov{e}_{(0,1,0)}-\ov{e}_2(J)]\\&+&
t[\ov{e}_{(0,1)}(I,K)+\ov{e}_{(0,1)}(J,K)-\ov{e}_{(0,0,1)}-\ov{e}_2(K)]\\&+&
\ov{e}_3(IJK)-[\ov{e}_3(IJ)+\ov{e}_3(IK)+\ov{e}_3(JK)]+\ov{e}_3(I)+\ov{e}_3(J)+\ov{e}_3(K).
\end{eqnarray*}
\end{theorem}
\begin{proof}
We have 
\begin{eqnarray*}
 \lm_R \left(\frac{\ov{I^rJ^sK^t}}{a^r\ov{J^sK^t}+b^s\ov{I^rK^t}+c^t\ov{I^rJ^s}}\right)
&=& \lm_R\left(\frac{R}{a^r\ov{J^sK^t}+b^s\ov{I^rK^t}+c^t\ov{I^rJ^s}}\right)-\lm_R\left(\frac{R}{\ov{I^rJ^sK^t}}\right)\\
&=&\lm_R\left(\frac{R}{(a^r,b^s,c^t)}\right)+\lm_R\left(\frac{(a^r,b^s,c^t)}{a^r\ov{J^sK^t}+b^s\ov{I^rK^t}+c^t\ov{I^rJ^s}}\right)-
\lm_R\left(\frac{R}{\ov{I^rJ^sK^t}}\right)\\
&=& rst \lm_R\left(\frac{R}{(a,b,c)}\right)+\left[\lm_R\left(\frac{R}{\ov{I^rJ^s}}\right)+
\lm_R\left(\frac{R}{\ov{I^rK^t}}\right)+\lm_R\left(\frac{R}{\ov{J^sK^t}}\right)\right]\\
&-&\left[\lm_R\left(\frac{R}{\ov{I^r}}\right)+\lm_R\left(\frac{R}{\ov{J^s}}\right)+\lm_R\left(\frac{R}{\ov{K^t}}\right)\right]
-\lm_R\left(\frac{R}{\ov{I^rJ^sK^t}}\right), \mbox{ by Proposition }\ref{computation of some length}.
\end{eqnarray*}
By \cite[Theorem 2.4]{rees3}, $\ov{e}_{(1,1,1)}=\lm_R\left(\frac{R}{(a,b,c)}\right)$ and 
\begin{center}
\begin{tabular}{|c|c|c|}
\hline$\ov{e}_{(3,0)}(I,J)=\ov{e}_{(3,0)}(I,K)=$  & 
$\ov{e}_{(0,3)}(I,J)=\ov{e}_{(3,0)}(J,K)=$ & 
$\ov{e}_{(0,3)}(I,K)=\ov{e}_{(0,3)}(J,K)=$\\
$\ov{e}_{(3,0,0)}=e(I)$ & $\ov{e}_{(0,3,0)}=e(J)$ & $\ov{e}_{(0,0,3)}=e(K)$\\
\hline
$\ov{e}_{(2,1,0)}=\ov{e}_{(2,1)}(I,J)$  & $\ov{e}_{(0,2,1)}=\ov{e}_{(2,1)}(J,K)$ & $\ov{e}_{(2,0,1)}=\ov{e}_{(2,1)}(I,K)$\\
\hline
$\ov{e}_{(1,2,0)}=\ov{e}_{(1,2)}(I,J)$  & $\ov{e}_{(0,1,2)}=\ov{e}_{(1,2)}(J,K)$ & $\ov{e}_{(1,0,2)}=\ov{e}_{(1,2)}(I,K)$ \\
\hline
\end{tabular}
\end{center}
Therefore, using Lemma \ref{lemma1}, for $r,s,t\gg0,$ we have
\begin{eqnarray*}
\lm_R \left(\frac{\ov{I^rJ^sK^t}}{a^r\ov{J^sK^t}+b^s\ov{I^rK^t}+c^t\ov{I^rJ^s}}\right)
&=&r[\ov{e}_{(1,0)}(I,J)+\ov{e}_{(1,0)}(I,K)-\ov{e}_{(1,0,0)}-\ov{e}_2(I)]\\&+&
s[\ov{e}_{(0,1)}(I,J)+\ov{e}_{(1,0)}(J,K)-\ov{e}_{(0,1,0)}-\ov{e}_2(J)]\\&+&
t[\ov{e}_{(0,1)}(I,K)+\ov{e}_{(0,1)}(J,K)-\ov{e}_{(0,0,1)}-\ov{e}_2(K)]\\&+&
\ov{e}_3(IJK)-[\ov{e}_3(IJ)+\ov{e}_3(IK)+\ov{e}_3(JK)]+\ov{e}_3(I)+\ov{e}_3(J)+\ov{e}_3(K).
\end{eqnarray*}
\end{proof}

\section{Characterization of joint reduction number zero in terms of Local Cohomology modules}
Let $H^i_K(M)$ denotes the local cohomology module of an $R$-module $M$ with support in an ideal $K$ of $R.$ 
Let $(R,\m)$ be an analytically unramified local ring of dimension $3$ and $I,J,K$ be 
$\m-$primary ideals in $R$. Let $a \in I,b \in J$ and $c \in K.$ 
Let   $\ov{{\mathcal R}^\prime}$ denote the extended Rees algebra of the filtration 
${\mathcal F}=\{ \ov{I^rJ^sK^t} \}.$
Consider the Koszul complex on $((at_1)^k,(bt_2)^k,(ct_3)^k)$ in 
 $\ov{\mathcal{R}^\prime}:$
 $${F^k}^. : 0 \longrightarrow \ov{\mathcal{R}^\prime}
\buildrel\over\longrightarrow \bigoplus_{i=1}^3 \ov{\mathcal{R}^\prime}(k{\bf e}_i) 
\buildrel\alpha_k\over\longrightarrow \bigoplus_{1 \leq i<j\leq 3}\ov{\mathcal{R}^\prime}(k({\bf e}_i+{\bf e}_j)) 
\buildrel\beta_k\over\longrightarrow \ov{\mathcal{R}^\prime}(k({\bf e}_1+{\bf e}_2+{\bf e}_3))
\longrightarrow
0.
$$
The twists are given so that maps are degree zero
maps. Then
\begin{equation} \label{eq1}
H^i_{(at_1,bt_2,ct_3)}(\ov{\mathcal{R}^\prime}) =
\displaystyle\lim_{\stackrel{\longrightarrow}{k}}
H^i({F^k}^{.})\end{equation}
for all
$i$ by  \cite[Theorem 5.2.9]{bs}. 
The purpose of this section is to find the length of the component at origin of the third local cohomology module $H^3_{(at_1,bt_2,ct_3)}(\ov{\mathcal{R}^\prime})$ in terms of normal Hilbert coefficients and joint reduction $(a,b,c).$ This leads to a satisfactory analogue of Rees's Theorem for $3$-dimensional analytically unramified local rings.
\begin{theorem}\label{localcohomoly}
 Let $(R,\m)$ be an analytically unramified Cohen-Macaulay local ring of dimension $3$ 
 with infinite residue field and $I,J,K$ be $\m-$primary ideals in $R$. 
 Let $(a,b,c)$ be a good joint reduction of 
 $\{\ov{I^rJ^sK^t}\}$. Then
\begin{enumerate}
 \item \label{localcohomoly part1} \noindent 
\begin{center}
$
[H^3_{(at_1,bt_2,ct_3)}(\ov{\mathcal{R}^\prime})]_{(0,0,0)} \cong
\displaystyle\lim_{\stackrel{\longrightarrow}{k}}
\frac{\ov{I^{k}J^{k}K^k}}{a^k\ov{J^{k}K^k}+b^k\ov{I^{k}K^k}+c^k\ov{I^kJ^k}}.
$ 
                         \end{center}
\item 
\label{localcohomoly part2}
For the directed system involved in the above direct limit, the map
$\mu_k$ 
$$\frac{\ov{I^{k}J^{k}K^k}}{a^k\ov{J^{k}K^k}+b^k\ov{I^{k}K^k}+c^k\ov{I^kJ^k}}
\buildrel\mu_k\over\longrightarrow
\frac{\ov{I^{k+1}J^{k+1}K^{k+1}}}{a^{k+1}\ov{J^{k+1}K^{k+1}}+b^{k+1}\ov{I^{k+1}K^{k+1}}+c^{k+1}\ov{I^{k+1}J^{k+1}}}$$ 
given by $\mu_k(x^\prime)=(xabc)^\prime$, where $x \in \ov{I^kJ^kK^k}$ 
and $^\prime$ denotes the image of an element in respect quotients, 
is injective for all $k>0$. 
\end{enumerate}
\end{theorem}
\begin{proof}
\begin{enumerate}
 \item 
The local cohomology module $H^3_{(at_1,bt_2,ct_3)}(\ov{\mathcal R^\prime})$ has a 
$\mathbb{Z}^3$-grading inherited from the $\mathbb{Z}^3$-grading of
$\ov{\mathcal{R}^\prime}$.
Therefore by the equation (\ref{eq1}),
$$
[H^3_{(at_1,bt_2,ct_3)}(\ov{\mathcal{R}^\prime})]_{(0,0,0)} =
\displaystyle\lim_{\stackrel{\longrightarrow}{k}}
\frac{\ov{I^{k}J^{k}K^k}}{(\image \beta_k)_{(0,0,0)}}.
$$
Since $(\image \beta_k)_{(0,0,0)}
\simeq[a^k\ov{J^{k}K^k}+b^k\ov{I^{k}K^k}+c^k\ov{I^kJ^k}],$
$$[H^3_{(at_1,bt_2,ct_3)}(\ov{\mathcal{R}^\prime})]_{(0,0,0)} \cong
\displaystyle\lim_{\stackrel{\longrightarrow}{k}}
\frac{\ov{I^{k}J^{k}K^k}}{a^k\ov{J^{k}K^k}+b^k\ov{I^{k}K^k}+c^k\ov{I^kJ^k}}.$$
\item Let $x \in \ov{I^kJ^kK^k}$ be such that $\mu_k(x^\prime)=0$. Then 
$xabc=a^{k+1}u+b^{k+1}v+c^{k+1}w$ for some $u \in \ov{J^{k+1}K^{k+1}}, v \in \ov{I^{k+1}K^{k+1}}$ and 
$w \in \ov{I^{k+1}J^{k+1}}$. Hence $u \in (b,c) \cap \ov{J^{k+1}K^{k+1}}=b\ov{J^kK^{k+1}}+c\ov{J^{k+1}K^k}$. 
Thus $u=bu_1+cu_2$ for some $u_1 \in \ov{J^kK^{k+1}}$ and $u_2 \in \ov{J^{k+1}K^k}$. Similarly, 
$v=av_1+cv_2, w=aw_1+bw_2$ for some $v_1\in \ov{I^kK^{k+1}},v_2 \in \ov{I^{k+1}K^k},
w_1 \in \ov{I^kJ^{k+1}}$ and $w_2 \in \ov{I^{k+1}J^k}$. 
Thus,
$$xabc=ac(a^{k}u_2+c^kw_1)+bc(b^{k}v_2+c^{k}w_2)+ab(a^ku_1+b^kv_1).$$
Hence $a^ku_1+b^kv_1 \in (c) \cap \ov{I^kJ^kK^{k+1}}=c\ov{I^kJ^kK^k}.$ Let $a^ku_1+b^kv_1=cp$ for some $p \in \ov{I^kJ^kK^k}$. Then 
$p \in (a^k,b^k) \cap \ov{I^kJ^kK^k}=a^k\ov{J^kK^k}+b^k\ov{I^kK^k}$. Similarly, $a^ku_2+c^kw_1=bq$ for some $q \in a^k\ov{J^kK^k}+c^k\ov{I^kJ^k}$ 
and $b^kv_2+c^kw_2=ar$ for some $r \in b^k\ov{I^kK^k}+c^k\ov{I^kJ^k}$. Hence 
$x=p+q+r \in a^k\ov{J^{k}K^k}+b^k\ov{I^{k}K^k}+c^k\ov{I^kJ^k}$. Thus $x^\prime=0$ and therefore $\mu_k$ is injective.
\end{enumerate}
\end{proof}
 
\begin{lemma}
\label{lemma on direct limit}
Let 
Let $(R,\m)$ be a Noetherian local ring. 
Let $\{(M_i,\psi_{i,j})\}_{i \geq 0}$ be a directed 
system of  $R-$modules and $M=\displaystyle\lim_{\stackrel{\longrightarrow}{k}} M_k$. \\
(1) If $M$ is a finitely generated $R$-module then  
the natural map $\psi_k: M_k\rightarrow M$ is 
surjective for all large $k$.
\\
(2) If
$\psi_{k,k+1}:M_k \longrightarrow M_{k+1} $ is injective 
for $k \gg 0$ then $\lm_k:M_k\rightarrow M$ is injective for large $k.$
\end{lemma}
\begin{proof} 
(1) Let $\lm_i : M_i \rightarrow \bigoplus_{i \geq 0}M_i$ be the 
ith injection. We know that $M= \frac{\bigoplus_k M_k}{N}$, where $N$ is the 
$R-$submodule of $\bigoplus_k M_k$ generated by elements 
$\lambda_j\psi_{i,j}(m_i)-\lambda_i(m_i)$ 
for $m_i \in M_i$ and $i \leq j$.  Let $\psi_k:M_k \longrightarrow M $ be the natural map defined as 
$\psi_k(m_k)=\lambda_k(m_k)+N$. Suppose that $M$ is generated by 
$\psi_{k_1}(x_{k_1}), \psi_{k_2}(x_{k_2}),\ldots, \psi_{k_g}(x_{k_g})$ 
for some $x_{k_i} \in M_{k_i}$.   
Let $k_0=\max\{k_1,k_2,\ldots,k_g\}$. 
Since $\psi_{k_i}=\psi_k\psi_{k_i,k}$ for $k\geq k_0$, 
$M$ is also generated by $\psi_k(\psi_{k_i,k}(x_{k_i}))$ for $i=1,2,\ldots,g.$ 
Hence $\psi_k: M_k\rightarrow M$ is surjective for $k\geq k_0$. \\
(2) Suppose $\psi_k(x)=0$ for some $x \in M_k$. Then $ \lm_k(x)+ N=0.$ 
By \cite[Lemma 5.30]{rotman}, $\psi_{k,j}(x)=0$ for some $j \geq k$. Hence $x=0$ if $k \gg 0$. 
\end{proof}

\begin{theorem} \label{theorem4}
Let $(R,\m)$ be an analytically unramified Cohen-Macaulay local ring 
of dimension $3$ with infinite residue field and $I,J,K$ be $\m-$primary ideals in $R$. 
Assume that a good complete reduction of the filtration $\{\ov{I^rJ^sK^t}\}$ exists. 
Let $(a,b,c)$ be any good joint reduction of the filtration $\{\ov{I^rJ^sK^t}\}$. Then
\begin{enumerate}
\item $\lm_R([H^3_{(at_1,bt_2,ct_3)}(\ov{\mathcal R^\prime})]_{(0,0,0)}) < \infty$.
 \item \label{theorem4 part2} $\lm_R([H^3_{(at_1,bt_2,ct_3)}(\ov{\mathcal R^\prime})]_{(0,0,0)})=\lm_R\left(\frac{\ov{I^{r}J^{s}K^t}}{a^r\ov{J^{s}K^t}+b^s\ov{I^{r}K^t}+c^t\ov{I^rJ^s}}\right)$ 
for large $r,s,t.$
\item \label{theorem4 part3} $\lm_R([H^3_{(at_1,bt_2,ct_3)}(\ov{\mathcal R^\prime})]_{(0,0,0)})=
\ov{e}_3(IJK)-[\ov{e}_3(IJ)+\ov{e}_3(IK)+\ov{e}_3(JK)]+\ov{e}_3(I)+\ov{e}_3(J)+\ov{e}_3(K)$.
\end{enumerate}
In particular, $\lm_R([H^3_{(at_1,bt_2,ct_3)}(\ov{\mathcal R^\prime})]_{(0,0,0)})$ is 
independent of good joint reduction $(a,b,c)$ of $\{\ov{I^rJ^sK^t}\}$.
\end{theorem}
\begin{proof}
\begin{enumerate}
\item By Theorem \ref{localcohomoly}(\ref{localcohomoly part1}), 
$$
[H^3_{(at_1,bt_2,ct_3)}(\ov{\mathcal{R}^\prime})]_{(0,0,0)} \cong
\displaystyle\lim_{\stackrel{\longrightarrow}{k}}
\frac{\ov{I^{k}J^{k}K^k}}{a^k\ov{J^{k}K^k}+b^k\ov{I^{k}K^k}+c^k\ov{I^kJ^k}}.
$$
Let $S_{(a,b,c)}(r,s,t)=\lm_R\left(\frac{\ov{I^{r}J^{s}K^t}}{a^r\ov{J^{s}K^t}+b^s\ov{I^{r}K^t}+
c^t\ov{I^rJ^s}}\right)$. First we show that $S_{(a,b,c)}(r,s,t)$ is a constant for $r,s,t \gg 0$. 
Let $(x_{ij})$ be a good complete reduction of $\{\ov{I^rJ^sK^t}\}$. 
Let $x=x_{11},y=x_{22}$ and $z=x_{33}$. Then, by Remark \ref{remark on gcr}(\ref{jr coming from gcr are good}), 
$(x,y,z)$ is a good joint reduction of $\{\ov{I^rJ^sK^t}\}$. By Theorem 
\ref{finiteness of local cohomology modules}, $\lm_R([H^3_{(xt_1,yt_2,zt_3)} 
(\ov{\mathcal R^\prime})]_{(0,0,0)})<\infty$. Thus $[H^3_{(xt_1,yt_2,zt_3)} 
(\ov{\mathcal R^\prime})]_{(0,0,0)}$ is a finitely generated $R-$module. Therefore, by Theorem 
\ref{localcohomoly} and Lemma \ref{lemma on direct limit}, for $k \gg 0$, 
\begin{eqnarray*}
 [H^3_{(xt_1,yt_2,zt_3)}(\ov{\mathcal{R}^\prime})]_{(0,0,0)} \cong 
 \frac{\ov{I^{k}J^{k}K^k}}{x^k\ov{J^{k}K^k}+y^k\ov{I^{k}K^k}+z^k\ov{I^kJ^k}}
\end{eqnarray*}
Hence
\begin{eqnarray}\label{Sk is constant}
\lm_R([H^3_{(xt_1,yt_2,zt_3)}(\ov{\mathcal R^\prime})]_{(0,0,0)})=S_{(x,y,z)}(k,k,k). 
\end{eqnarray}
By Theorem \ref{theorem1}, $S_{(x,y,z)}(r,s,t)$ is a polynomial function in $r,s,t$ of total 
degree atmost $1$. 
Fixing $s_0,t_0$ large, $S_{(x,y,z)}(r,s_0,t_0)$ is a polynomial function in $r$ of degree atmost $1$. 
Hence 
coefficient of $r$ in the polynomial associated with $S_{(x,y,z)}(r,s_0,t_0)$ is nonnegative. Thus 
the coefficient of $r$ in the polynomial associated with $S_{(x,y,z)}(r,s,t)$ is nonnegative. Similarly, 
coefficients of $s$ and $t$ in the polynomial associated with $S_{(x,y,z)}(r,s,t)$ are nonnegative. 
By equation (\ref{Sk is constant}), $S_{(x,y,z)}(k,k,k)$ is a constant for large $k$. 
Hence the sum of coefficients of $r,s,t$ in the polynomial 
associated with $S_{(x,y,z)}(r,s,t) $ is zero. Hence each coefficient of $r,s,t$ in the polynomial 
associated with $S_{(x,y,z)}(r,s,t) $ is zero. Thus, by Theorem \ref{theorem1}, for $r,s,t \gg 0,$
\begin{eqnarray} \label{formula for S(r,s,t)}
S_{(a,b,c)}(r,s,t)
=\ov{e}_3(IJK)-[\ov{e}_3(IJ)+\ov{e}_3(IK)+\ov{e}_3(JK)]+\ov{e}_3(I)+\ov{e}_3(J)+\ov{e}_3(K)
\end{eqnarray}
for any good joint reduction $(a,b,c)$ of $\{\ov{I^rJ^sK^t}\}$. 
Let 
$$\mu_k:\frac{\ov{I^{k}J^{k}K^k}}{a^k\ov{J^{k}K^k}+b^k\ov{I^{k}K^k}+c^k\ov{I^kJ^k}}
\longrightarrow
\frac{\ov{I^{k+1}J^{k+1}K^{k+1}}}{a^{k+1}\ov{J^{k+1}K^{k+1}}+b^{k+1}\ov{I^{k+1}K^{k+1}}+c^{k+1}\ov{I^{k+1}J^{k+1}}}
$$
be the map involved in the direct limit. 
By Theorem \ref{localcohomoly}(\ref{localcohomoly part2}), $\mu_k$ is injective for $k>0$. 
By equation (\ref{formula for S(r,s,t)}), $S_{(a,b,c)}(k,k,k)$ is a constant for $k \gg 0$. Hence 
$\mu_k$ is an isomorphism for $k \gg 0$. Therefore for $k \gg 0$, 
\begin{eqnarray}\label{formula for local cohomology}
 [H^3_{(at_1,bt_2,ct_3)}(\ov{\mathcal R^\prime})]_{(0,0,0)} \cong 
 \frac{\ov{I^{k}J^{k}K^k}}{a^k\ov{J^{k}K^k}+b^k\ov{I^{k}K^k}+c^k\ov{I^kJ^k}}
\end{eqnarray}
and hence $[H^3_{(at_1,bt_2,ct_3)}(\ov{\mathcal R^\prime})]_{(0,0,0)}$ has finite length.
\item 
Follows from equations (\ref{formula for S(r,s,t)}) and (\ref{formula for local cohomology}).

\item Follows from equations (\ref{formula for S(r,s,t)}) and (\ref{formula for local cohomology}).
\end{enumerate}
 \end{proof}

\begin{theorem}
\label{characterization of joint reduction number}
Let $(R,\m)$ be an analytically unramified Cohen-Macaulay local ring 
of dimension $3$ and $I,J,K$ be $\m-$primary ideals in $R$. 
Assume that a good complete reduction of the filtration $\{\ov{I^rJ^sK^t}\}$ exists. 
The following statements are equivalent:
\begin{enumerate}
\item $[H^3_{(at_1,bt_2,ct_3)}(\ov{\mathcal R^\prime})]_{(0,0,0)}=0$ for some good 
joint reduction $(a,b,c)$ of the filtration $\{\ov{I^rJ^sK^t}\}$,
\item $[H^3_{(at_1,bt_2,ct_3)}(\ov{\mathcal R^\prime})]_{(0,0,0)}=0$ for any good 
joint reduction $(a,b,c)$ of the filtration $\{\ov{I^rJ^sK^t}\}$,
 \item the normal joint reduction number of $I,J,K$ is zero with respect to any good joint 
 reduction $(a,b,c)$ of the filtration $\{\ov{I^rJ^sK^t}\}$,
 \item the normal joint reduction number of $I,J,K$ is zero with respect to some good joint 
 reduction $(a,b,c)$ of the filtration $\{\ov{I^rJ^sK^t}\}$,
\item $\ov{e}_3(IJK)-[\ov{e}_3(IJ)+\ov{e}_3(IK)+\ov{e}_3(JK)]+\ov{e}_3(I)+\ov{e}_3(J)+\ov{e}_3(K)=0.$
\end{enumerate}
 \end{theorem}
\begin{proof}
$(1) \Rightarrow (2):$ The result follows by the Theorem \ref{theorem4}(\ref{theorem4 part3}).\\  
$(2) \Rightarrow (3):$ By Theorem \ref{theorem4}(\ref{theorem4 part2}), for $r,s,t \gg0$, 
$$\ov{I^rJ^sK^t}=a^r\ov{J^sK^t}+b^s\ov{I^rK^t}+c^t\ov{I^rJ^s},$$
say for $r,s,t \geq N\geq 1$. Suppose $N \geq 2$. We prove that for $s,t \geq N$
$$\ov{I^{N-1}J^sK^t}=a^{N-1}\ov{J^sK^t}+b^s\ov{I^{N-1}K^t}+c^t\ov{I^{N-1}J^s}.$$
Let $z \in \ov{I^{N-1}J^sK^t}.$ Then $az \in \ov{I^NJ^sK^t}$. Let $az=a^Nu+b^sv+c^tw$ for some 
$u \in \ov{J^sK^t},v \in \ov{I^NK^t},w\in \ov{I^NJ^s}$. Then $z-a^{N-1}u \in (b^s,c^t) \cap \ov{I^{N-1}J^sK^t}=
b^s\ov{I^{N-1}K^t}+c^t\ov{I^{N-1}J^s}$, by Lemma \ref{properties of good joint reduction}. Hence $z \in a^{N-1}\ov{J^sK^t}+b^s\ov{I^{N-1}K^t}+c^t\ov{I^{N-1}J^s}$.\\
Similar argument shows that 
\begin{eqnarray*}
 \ov{I^{r}J^{N-1}K^t}&=&a^{r}\ov{J^{N-1}K^t}+b^{N-1}\ov{I^{r}K^t}+c^t\ov{I^rJ^{N-1}} \mbox{ for }r,t \geq N\\
 \ov{I^{r}J^sK^{N-1}}&=&a^{r}\ov{J^sK^{N-1}}+b^s\ov{I^{r}K^{N-1}}+c^{N-1}\ov{I^{r}J^s} \mbox{ for }r,s \geq N.
\end{eqnarray*}
Continuing as above, we get that for all $r,s,t \geq 1,$
$$\ov{I^rJ^sK^t}=a^r\ov{J^sK^t}+b^s\ov{I^rK^t}+c^t\ov{I^rJ^s}.$$
Hence $\ov{I^rJ^sK^t}\subseteq a\ov{I^{r-1}J^sK^t}+b\ov{I^rJ^{s-1}K^t}+c\ov{I^rJ^sK^{t-1}}$ for all 
$r,s,t \geq 1$. Thus the normal joint reduction number of $I,J,K$ is 
zero with respect to $(a,b,c)$.\\ 
$(3) \Rightarrow (4):$ This is clear.\\
$(4) \Rightarrow (5):$ First we prove that for $r,s,t>0$,
\begin{eqnarray}\label{power of a}
\ov{I^rJ^sK^t}=a^r\ov{J^sK^t}+b\ov{I^rJ^{s-1}K^t}+c\ov{I^rJ^sK^{t-1}}.
\end{eqnarray}
Induct on $r.$ Since the normal joint reduction number of $I,J,K$ is zero with respect to $(a,b,c)$, 
the result is true for $r=1$. Let $r>1$ and $x \in \ov{I^rJ^sK^t}$. Therefore 
$x=au+bv+cw$ for some $u \in \ov{I^{r-1}J^sK^t},v\in \ov{I^rJ^{s-1}K^t}, w \in \ov{I^rJ^sK^{t-1}}.$ 
By induction, $u=a^{r-1}u_1+bu_2+cu_3$ for some $u_1 \in \ov{J^sK^t},u_2\in\ov{I^{r-1}J^{s-1}K^t},u_3 \in 
\ov{I^{r-1}J^sK^{t-1}}.$ Hence $x \in a^r\ov{J^sK^t}+b\ov{I^rJ^{s-1}K^t}+c\ov{I^rJ^sK^{t-1}}$. 
Similarly, for $r,s,t>0$, 
\begin{eqnarray} 
\label{power of b} \ov{I^rJ^sK^t}&=&a\ov{I^{r-1}J^sK^t}+b^s\ov{I^rK^t}+c\ov{I^rJ^sK^{t-1}} \mbox{ and }\\
\label{power of c}\ov{I^rJ^sK^t}&=&a\ov{I^{r-1}J^sK^t}+b\ov{I^rJ^{s-1}K^t}+c^t\ov{I^rJ^s}.
\end{eqnarray}
We use induction on $r+s+t$ to prove that for $r,s,t>0$,
$$\ov{I^rJ^sK^t}=a^r\ov{J^sK^t}+b^s\ov{I^rK^t}+c^t\ov{I^rJ^s}.$$ 
Suppose $r+s+t=3$. Then $r=s=t=1$. Since the joint reduction number is zero, result is true in 
this case. Let $r+s+t>3$. Without loss of generality we may assume that $s>1$. 
If $r=t=1$, then result follows from the equation (\ref{power of b}). Therefore we may also assume that $t>1$. 
Let $x \in \ov{I^rJ^sK^t}$. Then, by the equation (\ref{power of a}), $x=a^ru+bv+cw$ for some $u \in \ov{J^sK^t},v \in \ov{I^rJ^{s-1}K^t}$ and 
$w \in \ov{I^rJ^sK^{t-1}}$. By induction, $v=a^rv_1+b^{s-1}v_2+c^tv_3$ and 
$w =a^rw_1+b^sw_2+c^{t-1}w_3$ for some $v_1 \in \ov{J^{s-1}K^t},v_2 \in \ov{I^rK^t}, v_3\in\ov{I^rJ^{s-1}}, 
w_1 \in \ov{J^sK^{t-1}},w_2 \in \ov{I^rK^{t-1}},w_3 \in \ov{I^rJ^s}$. Thus $x \in 
a^r\ov{J^sK^t}+b^s\ov{I^rK^t}+c^t\ov{I^rJ^s} $.\\
Therefore, by Theorem \ref{theorem1}, result follows.\\
$(5) \Rightarrow (1):$ Follows from Theorem \ref{theorem4}(\ref{theorem4 part3}).
\end{proof}

\begin{theorem}
\label{vanishing of e_3(IJK)}
Let $(R,\m)$ be an analytically unramified Cohen-Macaulay local ring 
of dimension $3$ and $I,J,K$ be $\m-$primary ideals in $R$. Let $(a,b,c)$ be a good joint 
reduction of the filtration $\{\ov{I^rJ^sK^t}\}$. Assume that 
$\lm_R(R/\ov{I^rJ^sK^t})=\ov{P}_{I,J,K}(r,s,t)$ for $r+s+t>0$. Then the 
normal joint reduction number of $I,J,K$ is zero with respect to 
$(a,b,c)$ if and only of $\ov{e}_3(IJK)=0$.
\end{theorem}
\begin{proof}
 Since $\lm_R(R/\ov{I^rJ^sK^t})=\ov{P}_{I,J,K}(r,s,t)$ for $r+s+t>0$, taking $s=t=0$ and $r>0$, we get 
$$\ov{e}_{(1,0,0)}=\ov{e}_2(I), \ov{e}_3(IJK)=\ov{e}_3(I).$$
Taking $t=0$ and $r,s \gg 0$, we get $\ov{e}_{(1,0,0)}=\ov{e}_{(1,0)}(I,J),\ov{e}_{(0,1,0)}=\ov{e}_{(0,1)}(I,J)$ and 
$\ov{e}_3(IJK)=\ov{e}_3(IJ)$. Since $\lm_R(R/\ov{I^rJ^sK^t})=\ov{P}_{I,J,K}(r,s,t)$ for $r+s+t>0$, 
arguing similarly and using Theorem \ref{theorem1},
we get that for $r,s,t > 0$, 
\begin{eqnarray*}
 \lm_R\left(\frac{\ov{I^rJ^sK^t}}{a^r\ov{J^sK^t}+b^s\ov{I^rK^t}+c^t\ov{I^rJ^s}}\right)
=\ov{e}_3(IJK).
\end{eqnarray*}
Thus $\ov{e}_3(IJK)=0$ if and only if for $r,s,t >0$, 
$\ov{I^rJ^sK^t}=a^r\ov{J^{s}K^t}+b^s\ov{I^{r}K^t}+c^t\ov{I^rJ^s}.$ 
Hence the 
normal joint reduction number of $I,J,K$ is zero with respect to 
$(a,b,c)$ if and only of $\ov{e}_3(IJK)=0$.
 \end{proof}

\section{Some Applications} 
In this section we discuss a few applications of Theorem \ref{characterization of joint reduction number}. 
Let $(R,\m)$ be an analytically unramified local ring of dimension $d$ and $I$ be an 
$\m-$primary ideal in $R$. An ideal $K \subseteq I$ is said to be a
{\it reduction} of the filtration $\{\ov{I^n}\}$ if $K\ov{I^n}=\ov{I^{n+1}}$ 
for all large $n$. A {\it minimal reduction} of $\{\ov{I^n}\}$ is a reduction of
$\{\ov{I^n}\}$ minimal with respect to inclusion. For a 
minimal reduction $K$ of $\{\ov{I^n}\}$, we set 
$$\ov{r}_K(I)= \sup \{n\in
\mathbb Z \mid \ov{I^n}\not= K\ov{I^{n-1}}\}.$$ 
The {\it reduction  number} $\ov{r}(I)$ of $\{\ov{I^n}\}$ is defined to be the
least $\ov{r}_K(I)$ over all possible minimal reductions 
$K$ of $\{\ov{I^n}\}$.

\begin{theorem}\cite[Corollary 3.17]{marleythesis}
 Let $(R,\m)$ be an analytically unramified Cohen-Macaulay local ring of dimension $d \geq 3$ 
 with infinite residue field and 
 $I $ be an $\m-$primary ideal in $R$. Let $\ov{G}(I)=\bigoplus_{n \geq 0}\ov{I^n}/\ov{I^{n+1}}$ 
 be the associated graded algebra of $\{\ov{I^n}\}$. 
 Suppose $\depth \ov{G}(I) \geq d-1$. Then $\ov{r}(I) \leq 2$ if and only if 
 $ \ov{e}_3(I)=0$. 
 \end{theorem}
\begin{proof}
Let $\ov{r}(I) \leq 2$. Then $\ov{e}_3(I)=0$ by \cite[Corollary 3.17]{marleythesis}. 
To prove the converse we use induction on $d$. Let $d=3$. 
Since $\depth \ov{G}(I) \geq 2, $ there exists a minimal reduction $J=(a,b,c)$ of $\{\ov{I^n}\}$ such that $a^*,b^*,c^{*}$ are 
nonzerodivisors in $\ov{G}(I)$ and $(a^*,b^*),(b^*,c^*),(a^*,c^*)$ are $\ov{G}(I)-$regular sequences. 
Hence 
$(x_{ij})$, 
where $x_{i1}=a,x_{i2}=b$ and $x_{i3}=c, $ for all $i=1,2,3$, is a good complete 
reduction of $\{\ov{I^rI^sI^t}\}$ and $(a,b,c)$ is a good joint reduction of $\{\ov{I^rI^sI^t}\}$.  
Since $\ov{e}_3(I)=\ov{e}_3(I^2)=\ov{e}_3(I^3)$, by Theorem 
\ref{characterization of joint reduction number}, $\ov{e}_3(I)=0$ if and only if $\ov{r}_{J}(I) \leq 2$. 
By \cite[Corollary 3.8]{marleythesis}, $\ov{r}(I)$ is independent of the minimal reduction chosen. 
Hence $\ov{e}_3(I)=0$ if and only if $\ov{r}(I) \leq 2$.
\\
Suppose $d>3$. Let $I=(a_1,\ldots,a_l)$ for some $l \geq 1$. For indeterminates $X_1,\ldots,X_l$, 
set $S=R[X_1,\ldots,X_l]_{\m[X_1,\ldots,X_l]},z=a_1X_1+\cdots+a_lX_l$ and $C=S/zC $. Then, 
by \cite[Theorem 1]{itoh} and \cite[Corollary 8]{itoh}, $C$ is an analytically unramified 
Cohen-Macaulay local ring of dimension $d-1$ and $\ov{e}_3(I)=\ov{e}_3(IC)$. 
Therefore then $\ov{r}(IC) \leq 2$, by induction. Hence, by \cite[Proposition 17]{itoh}, 
$\ov{r}(I) \leq 2$. 
\end{proof}

\begin{theorem}
\label{negative postulation number}
Let $(R,\m)$ be a Cohen-Macaulay local ring of dimension $d$ and ${\bf I}=(I_1,I_2,I_3)$ be 
$\m-$primary ideals in $R$. Let $\mathcal F$ be an ${\bf I}-$admissible filtration. Suppose 
that $\mathcal R(\mathcal F)$ is Cohen-Macaulay. Then $P_{\mathcal F}({\bf n})=H_{\mathcal F}({\bf n})$ 
for all ${\bf n} \geq {\bf 0}$.
\end{theorem}
\begin{proof}
Since $a^j(\mathcal R(\mathcal F))=-1$, by Lemma \ref{negativity of a-invariants}, 
for each $j=1,2,3$
$$[H^i_{\mathcal M}(\mathcal R(\mathcal F))]_{\bf n}=0\mbox{ for all } n_j \geq  0  
\mbox{ and all }i \geq 0.
$$
Therefore, by Lemma \ref{vanishing of lc T++}, 
$[H^i_{\mathcal R_{++}}(\mathcal R(\mathcal F))]_{\bf n}=0\mbox{ for all } {\bf n} \geq  {\bf 0}  
\mbox{ and all }i \geq 0.
$
Thus, by Lemma \ref{local cohomology of R and R*}, 
$[H^i_{\mathcal R_{++}}(\mathcal R^\prime(\mathcal F))]_{\bf n}=0\mbox{ for all } {\bf n} \geq  {\bf 0}  
\mbox{ and all }i \geq 0.
$
Hence, by difference formula \cite[Theorem 5.1]{jayanthan-verma}, for all ${\bf n} \geq {\bf 0}$, we have
\begin{eqnarray*}
P_{\mathcal F}({\bf n})-H_{\mathcal F}({\bf n})&=&\sum_{i=0}^d(-1)^i\lm_R([H^i_{\mathcal R_{++}}
(\mathcal R^\prime(\mathcal F))]_{\bf n})=0
 \end{eqnarray*}
Hence $P_{\mathcal F}({\bf n})=H_{\mathcal F}({\bf n})$ for all ${\bf n} \geq  {\bf 0}$. 
\end{proof}

\begin{theorem}
\label{joint reduction number of monomial ideals is zero}
 Let $R=k[x,y,z]$ be polynomial ring in $3$ variables $x,y,z$. Let $\m=(x,y,z)$ and $I,J,K$ be an $\m-$primary monomial ideal in $R$. 
 Let $S=R_\m$. Then the normal joint 
 reduction number of $IS,JS,KS$ is zero with respect to every good joint reduction of $\{\ov{I^rJ^sK^tS}\}$.
 \end{theorem}
\begin{proof}
 Note that $S$ is an analytically unramified 
 Cohen-Macaulay local ring of dimension $3$ and $IS,JS,KS$ are $\m S-$primary ideals in $S$. 
 By \cite[Theorem 6.3.5]{bruns-herzog}, $\ov{\mathcal R}(I,J,K)$ is Cohen-Macaulay. Since 
 $\ov{\mathfrak a S}=\ov{\mathfrak a}S$ for every ideal $\mathfrak{a}$ in $R$, we have
 $$\ov{\mathcal R}(IS,JS,KS)=\bigoplus_{r,s,t \geq 0}
 \ov{I^rJ^sK^tS}t_1^rt_2^st_3^t=\bigoplus_{r,s,t \geq 0} 
 \ov{I^rJ^sK^t}St_1^rt_2^st_3^t=\ov{\mathcal R}(I,J,K) S.$$ 
 Hence $\ov{\mathcal R}(IS,JS,KS)$ is Cohen-Macaulay. Thus, by Theorem 
 \ref{good joint reductions exist in CMcase}, 
 there exists a good complete reduction $(x_{ij})_{1\leq i,j\leq 3}$ of the filtration  
$\{\ov{I^rJ^sK^tS}\}$. 
 By Theorem \ref{negative postulation number}, $\ov{P}_{IS,JS,KS}(r,s,t)=\ov{H}_{IS,JS,KS}(r,s,t)$ 
 for all $r,s,t \geq 0$. Hence $\ov{P}_{IS,JS,KS}(0,0,0)=\ov{e}_3(IJKS)=0$. 
 Therefore, by Theorem \ref{vanishing of e_3(IJK)}, 
the normal joint reduction number of $IS,JS,KS$ is zero 
 with respect to every good joint reduction $(a,b,c)$ of $\{\ov{I^rJ^sK^tS}\}$. 
\end{proof}
 
\begin{theorem}
\label{one ideal}
 Let $R=k[x,y,z], \m=(x,y,z)$ and $I$ be an $\m-$primary monomial ideal in $R$. Suppose 
 $I $ and $I^2$ are complete. Then $I^n$ is complete for all $n \geq 1$.
 \end{theorem}
\begin{proof}
Let $S=R_\m$. 
Then, by Theorem \ref{joint reduction number of monomial ideals is zero}, the 
normal joint reduction number of $IS,IS,IS$ is zero with respect to 
every good joint reduction $J=(a,b,c)$ of $\{\ov{I^rI^sI^tS}\}$. 
Hence $\ov{I^nS}=J\ov{I^{n-1}S}$ for all $n \geq 3$. 
Since $I$ and $I^2$ are complete, $IS$ and $I^2S$ are compete. 
Thus $I^nS$ is complete for all $n \geq 1$. Hence $I^n$ is complete for 
all $n \geq 1$.
\end{proof}
 
\begin{theorem}
\label{two ideal}
 Let $R=k[x,y,z],\m=(x,y,z)$ and $I,J$ be $\m-$primary monomial ideals in $R$. Suppose that 
 $I^rJ^s$ is complete for all $r,s \geq 0$ such that $r+s \leq 2$. Then $I^rJ^s$ is complete for all 
 $r,s \geq 0$. 
\end{theorem}
\begin{proof}
Let $S=R_\m.$ 
By Theorem \ref{joint reduction number of monomial ideals is zero}, the normal joint reduction 
number of $IS,IS,JS$ and $IS,JS,JS$ is zero with respect to every good joint reduction of 
$\{\ov{I^rI^sJ^tS}\}$ and $\{\ov{I^rJ^sJ^tS}\}$ respectively. Let  
$(a_1,b_1,c_1)$ be a good joint reduction of the filtration $\{\ov{I^rI^sJ^tS}\}$ 
and $(a_2,b_2,c_2)$ be a good joint reduction of the filtration $\{\ov{I^rJ^sJ^tS}\}$. Then
 \begin{eqnarray} \label{two ideal equation1}
\ov{I^rJ^sS}&=&(a_1,b_1)\ov{I^{r-1}J^sS}+c_1\ov{I^rJ^{s-1}S} \mbox{ for all } (r,s)\geq (2,1)\mbox{ and }\\
\label{two ideal equation2}\ov{I^rJ^sS}&=&(a_2)\ov{I^{r-1}J^sS}+(b_2,c_2)\ov{I^rJ^{s-1}S} \mbox{ for all } (r,s)\geq (1,2).
\end{eqnarray}
It suffices to show that $I^rJ^sS$ is complete for all $r,s \geq 0$. 
By Theorem \ref{one ideal}, $I^nS$ and $J^nS$ are complete for all $n \geq 1$. 
We use induction on $r+s$ to show that $I^rJ^sS$ is complete for all $r,s \geq 0$. By assumption, 
result is true for $r+s \leq 2$. Suppose $r+s >2$. We may assume that $r>1$. Since $I^rS$ is complete 
for all $r \geq 1$, assertion is true for $s=0$. Hence we may also assume that $s \geq 1$. 
By induction, $I^{r-1}J^sS$ and $I^rJ^{s-1}S$ are complete. Therefore,
by equation (\ref{two ideal equation1}), $I^rJ^sS$ is complete. 
 \end{proof}

\begin{theorem}
 Let $R=k[x,y,z],\m=(x,y,z)$ and $I,J,K$ be $\m-$primary monomial ideals in $R$. Suppose that 
 $I^rJ^sK^t$ is complete for all $r,s,t \geq 0$ such that $r+s+t \leq 2$. Then $I^rJ^sK^t$ is complete 
 for all $r,s,t \geq 0$.
\end{theorem}
\begin{proof}
Let $S=R_\m$. 
By Theorem \ref{joint reduction number of monomial ideals is zero}, the normal joint reduction 
 number of $IS,JS,KS$ is zero with respect to every good joint reduction of the filtration 
 $\{\ov{I^rJ^sK^tS}\}$.  Let  
$(a,b,c)$ be a good joint reduction of the filtration $\{\ov{I^rJ^sK^tS}\}$. 
Hence, for all $r,s, t \geq 1, $ 
 \begin{eqnarray}\label{three ideal}
 \ov{I^rJ^sK^tS}=a\ov{I^{r-1}J^sK^tS}+b\ov{I^rJ^{s-1}K^tS}+c\ov{I^rJ^sK^{t-1}S}.
 \end{eqnarray}
 It suffices to show that $I^rJ^sK^tS$ is complete for all $r,s,t \geq 0$. 
 We use induction on $r+s+t$ to show that $I^rJ^sK^tS$ is complete for all $r,s,t \geq 0$. 
By assumption result is true for $r+s+t \leq 2$. Let $r+s+t\geq 3$. By Theorem \ref{two ideal}, 
$I^rJ^sS,I^rK^tS$ and $J^sK^tS$ are complete for all 
 $r,s,t\geq 0$. Hence result is true if $r=0$ or $s=0$ or $t=0$. Therefore, we may assume that $r,s,t 
 \geq 1$. By induction, $I^{r-1}J^sK^tJS, I^rJ^{s-1}K^tS$ and $I^rJ^sK^{t-1}S$ are complete. Hence, by equation 
 (\ref{three ideal}), $I^rJ^sK^tS$ is complete.
\end{proof}

\end{document}